\PassOptionsToPackage{dvipsnames}{xcolor}
\documentclass[reqno,12pt]{amsart}

\usepackage[lite]{amsrefs}

\usepackage{amssymb}

\usepackage[all,cmtip]{xy}

\usepackage[inline]{enumitem}
\usepackage{mathtools}
\usepackage{bm}
\usepackage{tkz-euclide}
\usepackage{tikz,xifthen}
\usepackage{tikz-cd}
\usepackage{mathrsfs}
\usetikzlibrary{arrows}
\usetikzlibrary{decorations.markings}
\usetikzlibrary{backgrounds}
\usepackage{pgfplots}
\pgfplotsset{compat=1.10}
\usepgfplotslibrary{fillbetween}
\usetikzlibrary{patterns}
\usepackage[T1]{fontenc}
\usepackage{tensor}
\usepackage{lmodern}
\usepackage{hyperref}
\usepackage[dvipsnames]{xcolor}

\hypersetup{
    colorlinks=true,
    linkcolor=MidnightBlue,
    filecolor=magenta,      
    urlcolor=MidnightBlue,
    pdftitle={Fundamentals of Lie categories},
    pdfpagemode=FullScreen,
}

\makeatletter
\renewcommand\subsection{\@startsection{subsection}{2}%
  \z@{-.5\linespacing\@plus-.7\linespacing}{.5\linespacing}%
  {\bf}}
\makeatother

\usetikzlibrary{decorations.markings}
\tikzset{double line with arrow/.style args={#1,#2}{decorate,decoration={markings,%
mark=at position 0 with {\coordinate (ta-base-1) at (0,1pt);
\coordinate (ta-base-2) at (0,-1pt);},
mark=at position 1 with {\draw[#1] (ta-base-1) -- (0,1pt);
\draw[#2] (ta-base-2) -- (0,-1pt);
}}}}
\tikzset{Equals/.style={-,double line with arrow={-,-}}}

\newcommand{\plus}{\scalebox{0.6}{$\bm{+}$}}

\makeatletter
\def\blfootnote{\xdef\@thefnmark{}\@footnotetext}
\makeatother
\setlist[enumerate]{leftmargin=25pt, label={(\roman*)}}

\hypersetup{colorlinks=true,linkcolor=MidnightBlue,citecolor=BrickRed}

\DeclarePairedDelimiterX\set[1]\lbrace\rbrace{\def\given{\;\delimsize\vert\;}#1}
\renewcommand{\d}[1]{\ensuremath{\operatorname{d}\!{#1}}}

\newcommand{\rra}{\rightrightarrows}

\newcommand{\R}{\mathbb{R}}

\newcommand{\rk}{\mathcal{R}}
\newcommand{\C}{\mathcal{C}}
\newcommand{\D}{\mathcal{D}}
\newcommand{\G}{\mathcal{G}}
\newcommand{\X}{\mathcal{X}}

\renewcommand{\Im}{\operatorname{Im}}
\newcommand{\End}{\operatorname{End}}
\newcommand{\Int}{\operatorname{Int}}
\newcommand{\GL}{\operatorname{GL}}
\newcommand{\rank}{\operatorname{rank}}
\newcommand{\rankl}{\operatorname{rank}^L}
\newcommand{\rankr}{\operatorname{rank}^R}
\newcommand{\codim}{\operatorname{codim}}
\newcommand{\ind}{\operatorname{ind}}
\newcommand{\pr}{\mathrm{pr}}
\newcommand{\id}{\mathrm{id}}
\newcommand{\inv}{\mathrm{inv}}

\newcommand{\vf}{\mathfrak X}

\newcommand{\comp}[1]{{#1}^{(2)}}

\DeclareRobustCommand{\gobblefive}[5]{}
\newcommand*{\SkipTocEntry}{\addtocontents{toc}{\gobblefive}}

\newtheorem*{claim}{Claim}
\numberwithin{equation}{section}

\theoremstyle{plain} 
\newtheorem{thm}[equation]{Theorem}
\newtheorem{cor}[equation]{Corollary}
\newtheorem{lem}[equation]{Lemma}
\newtheorem{prop}[equation]{Proposition}

\theoremstyle{definition}
\newtheorem{defn}[equation]{Definition}

\theoremstyle{remark}
\newtheorem{rem}[equation]{Remark}
\newtheorem*{rem*}{Remark}
\newtheorem{ex}[equation]{Example}

\pgfdeclarepatternformonly{mydots}{\pgfqpoint{-1pt}{-1pt}}{\pgfqpoint{1pt}{1pt}}{\pgfqpoint{1.5pt}{1.5pt}}
{%
  \pgfpathcircle{\pgfqpoint{0pt}{0pt}}{.3pt}%
  \pgfusepath{fill}%
}%

\calclayout

\newcommand{\Addresses}{{
  \bigskip
  \footnotesize

  Ž.\ Grad, \textsc{Centre for Mathematical Analysis, Geometry, and Dynamical Systems,
Department of mathematics,
Instituto Superior Técnico,
Av. Rovisco Pais,
Lisbon,
Portugal.}\par\nopagebreak
  \textit{E-mail address}: \url{zan.grad@tecnico.ulisboa.pt}.\par\nopagebreak
  
  \textit{Webpage}: \url{https://zangrad.github.io}.
}}

\begin{document}

\title{Fundamentals of Lie categories}

\author[Žan Grad]{Žan Grad}





\thanks{This work was supported in part by PhD Grant UI/BD/152069/2021 of FCT, Portugal.}

\begin{abstract}
We introduce the basic notions and present examples and results on Lie categories -- categories internal to the category of smooth manifolds. Demonstrating how the units of a Lie category $\C$ dictate the behavior of its invertible morphisms $\G(\C)$, we develop sufficient conditions for $\G(\C)$ to form a Lie groupoid. We show that the construction of Lie algebroids from the theory of Lie groupoids carries through, and ask  when the Lie algebroid of $\G(\C)$ is recovered. We reveal that the lack of invertibility assumption on morphisms leads to a natural generalization of rank from linear algebra, develop its general properties, and show how the existence of an extension $\C\hookrightarrow \G$ of a Lie category to a Lie groupoid affects the ranks of morphisms and the algebroids of $\C$. Furthermore, certain completeness results for invariant vector fields on Lie monoids and Lie categories with well-behaved boundaries are obtained. Interpreting the developed framework in the context of physical processes, we yield a rigorous approach to the theory of statistical thermodynamics by observing that entropy change, associated to a physical process, is a functor.
\end{abstract}

\maketitle

\vspace{-1em}
\tableofcontents

\blfootnote{
}
\vspace{-2em}

\section{Introduction}
\label{sec:intro}
Since its conception, category theory has proven to provide a unified framework for the language of mathematics, by making use of the observation that objects and morphisms thereof arise regardless of the mathematical field one considers. This paradigm has also been adopted by physicists (see e.g.\ \cite{catsquantum}), namely that one can interpret physical states (contained in a certain phase space) as objects, and  physical processes as morphisms between them, thus forming a category which corresponds to the physical system at hand. In order to provide such a realization of a given physical system, the corresponding category should ideally have the capacity to straightforwardly describe the phenomena which pertain to the given physical system, and also capture the means for describing and calculating relevant physical quantities. For a physicist, the latter is generally done using the basic tools of calculus, in terms of a preferred set of coordinates. 

There is a natural way of obtaining such a unified framework for describing physical processes, by intertwining category theory with the theory of smooth manifolds. That is, we require the set of objects (states) and the set of morphisms (processes) of an abstract category to be a pair of smooth manifolds, and that the composition law $(f,g)\mapsto fg$ is a smooth map (concatenation of two processes); the mathematical structure obtained is that of a category internal to the category $\mathbf{Diff}$ of smooth manifolds. Historically, this kind of a mathematical structure, which we will define in more precise terms as a Lie category, was first introduced and briefly studied by Charles Ehresmann in his seminal paper \cite{ehr1959}; in the same paper, Ehresmann further focused on the notion of a Lie groupoid, which additionally imposes that all morphisms are invertible. This latter notion has nowadays been thoroughly researched and continues to have a status of an active field of research, whereas the same cannot be said for Lie categories; apart from the work of Ehresmann, this is the first paper dedicated to their systematic treatment. The main reason behind the fact that Lie categories have gained negligible attention compared to Lie groupoids lies in the fact that the assumption on morphisms being invertible implies that all left and right translations are diffeomorphisms, which accounts for certain preferable properties of the structure of a Lie groupoid, that will also be highlighted in this paper.

Our motivation for studying Lie categories, and not merely groupoids, stems from the physical interpretation that invertible morphisms correspond to reversible processes, and in physics not all processes in a given system are reversible; this is well known from the theory of thermodynamics, where reversible processes are precisely those where the change of entropy when transiting between two states is zero; we will describe this in precise terms in Section \ref{sec:std}. Another, perhaps more notorious motivation, is given by irreversibility of wave-function collapse in quantum theory, which was actually our initial motivation for dropping the invertibility assumption.

As we will see, it will turn out to be desirable to allow the space of morphisms of a Lie category to possess a boundary -- we will encounter both, mathematical and physical examples where the boundary of the space of morphisms will play a distinguished role. To temporarily appease and motivate the reader to this regard, let us briefly note that Lie monoids are examples of Lie categories with the set of objects a singleton, and that the monoid $[0,\infty)$ -- either for multiplication or addition -- provides a first example of a Lie category with boundary. This simple example already shows certain intriguing qualities: for instance, considering $[0,\infty)$ for multiplication, all its invertible elements are contained in its interior, and considering $[0,\infty)$ for addition, its only invertible element is in its boundary. This phenomenon, as may be expected, is one of the features of Lie categories, namely that the units dictate the behavior of invertibles.

Overall, we aim to convince the reader that the interplay of geometrical and categorical structures alone (without the invertibility assumption) provides new exciting questions which were so far overlooked within the scope of Lie theory. Let us summarize our main objectives.
\begin{enumerate}
\item To demonstrate that Lie categories allow for an abundance of interesting examples which have so far been missed in the theory of Lie groupoids.
\item To expose some ideas and constructions which carry through to Lie categories from the theory of Lie groupoids, e.g.\ the Lie algebroid construction.
\item To inspect the relation between Lie groupoids and Lie categories, and show that novel notions can be obtained when invertibility is dropped.
\item Last, but not least, to provide an algorithm for constructing Lie categories that describe physical systems by making use of the notion of entropy, and to reveal that the mathematical structure implicitly present in the foundations of statistical thermodynamics is that of a Lie category.
\end{enumerate}
\SkipTocEntry\section*{Notation}

All our categories are small, i.e.\ the classes of objects and morphisms are sets. We will denote a category by $\C\rra \X$, where $\C$ is the set of morphisms, $\X$ is the set of objects, and the two arrows indicate the \textit{source map} $s\colon\C\rightarrow \X$ and \textit{target map} $t\colon\C\rightarrow\X$, which are defined by 
\[
s(x\xrightarrow g y)=x,\quad t(x\xrightarrow g y)=y,
\]
for any morphism $g\colon x\rightarrow y$ in $\C$. Alongside these two maps, a category $\C$ comes equipped with the \textit{composition map}
\[
m\colon \comp\C\rightarrow \C,\quad (g,h)\mapsto gh,
\]
where $\comp\C=\set{(g,h)\in \C\times \C\given s(g)=t(h)}$ is the set of all pairs of \textit{composable morphisms}. Moreover, $\C$ comes equipped with the \textit{unit map}
\[
u\colon \X\rightarrow \C,\quad x\mapsto 1_x.
\]
We also define
\[
\C_x=s^{-1}(x),\quad \C^y=t^{-1}(y),\quad \C_x^y=\C_x\cap \C^y,
\]
and call $\C_x$ the \textit{source fibre} at $x$, and $\C^y$ the \textit{target fibre} at $y$. Note that any morphism $g\in\C$ determines the maps
\begin{equation}
\label{eq:trans}
\begin{aligned}
&L_g\colon \C^{s(g)}\rightarrow \C^{t(g)},\quad L_g(h)=gh,\\
&R_g\colon \C_{t(g)}\rightarrow \C_{s(g)},\quad R_g(h)=hg,
\end{aligned}
\end{equation}
called the \textit{left translation} and \textit{right translation} by $g$, which are just the pre-composition and post-composition by $g$, respectively.

\section{Basic definitions and examples}
\label{sec:basics}

\begin{defn}
A \textit{Lie category} is a small category $\C \rra \X$, where $\C$ is a smooth manifold with or without boundary, $\X$ is a smooth manifold without boundary, and there holds:
\begin{enumerate}[label={(\roman*)}]
\item The source and target maps $s,t\colon \C\rightarrow \X$ are smooth submersions.
\item The unit map $u\colon \X\rightarrow \C$ and the composition map $m\colon \comp\C\rightarrow \C$ are smooth.
\end{enumerate}
If $\C$ has a boundary, we also assume that $\C\rra\X$ has a \textit{regular boundary}, that is:
\begin{enumerate}
\item[(iii)] The restrictions $\partial s,\partial t\colon\partial \C\rightarrow \X$ of $s$ and $t$ are smooth submersions.
\end{enumerate}
\end{defn}
\begin{rem}
Given any $x\in\X$, assumptions (i) and (iii) ensure that $\C_x$ and $\C^x$ are neat submanifolds of $\C$ (see \cite[p.\ 60]{difftop} or \cite[Proposition 4.2.9]{corners}), that is:
\begin{align}
\label{eq:fibrebdr}
\partial(\C_x)=\C_x\cap \partial\C \quad\text{and}\quad \partial(\C^x)=\C^x\cap \partial\C.
\end{align}
Moreover, assumptions (i) and (iii) ensure that the set
\[
\comp\C=(s\times t)^{-1}(\Delta_\X)
\]
of composable morphisms is a neat submanifold of $\C\times\C$, that is:
\begin{align}
\label{eq:regbdr_c2}
\partial(\comp\C)=\comp\C\cap \partial(\C\times\C)=\comp \C\cap (\C\times \partial\C \cup \partial\C\times \C),
\end{align}
by transversality theorem, see Corollary \ref{cor:fibred_prod}. This corollary also ensures that if $\C$ has a boundary, the corner points of $\comp\C$ are precisely the composable pairs in $\partial\C\times\partial\C$; moreover, the tangent space of $\comp\C$ at a composable pair $(g,h)$ equals
\begin{align}
T_{(g,h)}\C^{(2)}=\set{(v,w)\in T_g\C\oplus T_h\C\given \d s(v)=\d t(w)}.
\end{align}
Smoothness of $\comp\C$ implies that the requirement of smoothness of the composition map $m\colon \comp\C\rightarrow \C$ makes sense, and furthermore that left and right translations $L_g,R_g$ are smooth maps between appropriate fibres, as defined by equations \eqref{eq:trans}; we thus obtain a covariant functor $\C\rightarrow \mathbf{Diff}$, given on objects as $x\mapsto \C^x$ and on morphisms as $g\mapsto L_g$, and a contravariant one given by $x\mapsto \C_x$, $g\mapsto R_g$. 
\end{rem}
\begin{rem}
\label{rem:units_embedded}
The unit map $u\colon \X\rightarrow \C$ of a Lie category $\C\rra\X$ is an embedding, which is a consequence of the fact that it is a smooth section of the source (and target) map, hence an injective immersion which is a homeomorphism onto its image, whose continuous inverse is given by $s|_{u(\X)}$. 
\end{rem}
\begin{defn}
A \textit{morphism} of Lie categories is a smooth functor $F\colon \C\rightarrow \D$. A Lie category $\C$ is said to be a \textit{Lie subcategory} of $\D$, if there is an injective immersive morphism $F\colon\C\rightarrow \D$ of Lie categories.
\end{defn}

Remark \ref{rem:units_embedded} implies that any morphism of Lie categories induces a smooth map between the respective object manifolds of $\C$ and $\D$. We now turn to examples of Lie categories.

\begin{ex}\
In the case when the object space $\X$ is a singleton, a Lie category $\C\rra \set{*}$ will be called a \textit{Lie monoid}. Simply put, a Lie monoid is a monoid $M$ together with a structure of a smooth manifold with or without boundary, such that the multiplication $m\colon M\times M\rightarrow M$ is smooth.

Concrete examples of Lie monoids frequently arise as embedded submonoids of Lie groups. For instance, we may consider the closed ray $[0,\infty)\subset \R$ for addition, or more generally, the $n$-dimensional half-space $\mathbb H^n=\R^{n-1}\times [0,\infty)\subset \R^n$, where $n\in\mathbb N$, which is a commutative Lie monoid for the usual addition. On the other hand, the closed ray $[0,\infty)$ for multiplication is not\footnote{This is easily seen since it has a non-cancellative (absorbing) element.} a submonoid of any Lie group. Further examples of Lie monoids that do not arise as submonoids of Lie groups are: 
\begin{enumerate}[label={(\roman*)}]
\item The set $\R^{n\times n}$ of square $n$-dimensional matrices is a Lie monoid for matrix multiplication; more generally, we may consider the Lie monoid $\End(V)$ of endomorphisms of a finite dimensional vector space $V$, under composition. Even more generally, any finite-dimensional unital algebra is a Lie monoid that is enriched over the category $\mathbf{Vect}$ of vector spaces, since the multiplication map is bilinear.
In particular, this includes the real line, the complex plane, and quaternions for multiplication.
\item The closed unit disk $\overline{\mathbb D}\subset \mathbb C$, an abelian Lie monoid for complex multiplication, and the closed unit 4-ball $\bar{\mathbb B}^4$, a non-abelian Lie monoid for quaternionic multiplication.
\end{enumerate}
\end{ex}

The following example generalizes the similar notion of triviality from the theory of Lie groupoids.
\begin{ex}
Let $X$ be a smooth manifold without boundary and $M$ a Lie monoid. A \textit{trivial Lie category} is defined by $\C=X\times M\times X$ and $\X=X$, with $s=\pr_3, t=\pr_1$ and composition as
\[
(z,g,y)(y,h,x)=(z,gh,x).
\]
That composition is smooth follows from the smoothness of multiplication in $M$. In the case when $M=\set e$ is a trivial monoid, we obtain the well-known \textit{pair groupoid}.
\end{ex}

The next example reveals the spirit of the notion of a Lie category -- that is, it can be thought of as a smooth family of endomorphisms of an abstract structure, parametrized by the base manifold $\X$. This is aligned with the philosophy that a Lie groupoid can be thought of as a smooth collection of automorphisms (symmetries) of a structure parametrized by $\X$.
\begin{ex}
Let $\pi\colon E\rightarrow X$ be a vector bundle over a smooth manifold $X$  without boundary, whose typical fibre is a fixed vector space $V$. The \textit{endomorphism category} of $E\rightarrow X$ is the category $\End(E)\rra X$, where the set of morphisms is defined as the set
\[
\End(E)=\set{\xi\colon E_x\rightarrow E_y\given \xi \text{ is linear},x,y\in X}
\]
of linear homomorphisms between the fibres of $E\rightarrow X$, and the structure maps $s,t,m,u$ are defined in the obvious way. To show that $\End(E)$ admits a structure of a Lie category without boundary, we must define a smooth structure on $\End(E)$; it is induced by local trivializations on $E$ in the following way. Denote by 
\[
\set{U_i\times V\xrightarrow{\psi_i} \pi^{-1}(U_i)\given {i\in I}}
\]
an atlas of local trivializations of $E$ over an open cover $(U_i)_{i\in I}$ of $X$, and denote by $\tau_{ij}\colon U_i\cap U_j\rightarrow \GL(V)$ the respective transition maps, i.e.\ $\psi_i^{-1}\psi_j(x,v)=(x,\tau_{ij}(x)v)$. For any two indices $i,j\in I$, we define the map
\begin{align*}
&\Psi_i^j\colon U_j\times \End(V)\times U_i\rightarrow \End(E)^{U_j}_{U_i}=s^{-1}(U_i)\cap t^{-1}(U_j)\\
&\Psi_i^j(y,A,x)(e|_x)=\psi_j(y,A\pr_V \psi^{-1}_i(e|_x)),
\end{align*}
whose inverse is 
\[
(\Psi_i^j)^{-1}(\xi\colon E_x\rightarrow E_y)=(y,v\mapsto \pr_V\psi_j^{-1}\xi\psi_i(x,v),x).
\]
The smoothness of transition maps in this atlas is easily checked by computing \begin{align*}
(\Psi_k^l)^{-1}\Psi_i^j(y,A,x)=(y,\tau_{jl}(y)^{-1}A\tau_{ik}(x),x),
\end{align*}
where $x\in U_i\cap U_k$ and $y\in U_j\cap U_l$. From the local charts, it is clear that $s$ and $t$ are submersions, and the smoothness of composition map $m$ follows from smoothness of multiplication in $\End(V)$; finally, this composition is bilinear when restricted to $\End(E)^z_y\times \End(E)^y_x\subset \End(E)\comp{}$, so $\End(E)$ is moreover enriched over the category $\mathbf{Vect}$ of vector spaces.
\end{ex}

\begin{ex}[Bundles of Lie monoids]
A Lie category with coinciding source and target map $s=t=:p$ is called a \textit{bundle of Lie monoids}. In this case, two morphisms are composable if, and only if, they are in the same fibre of $p$.

A concrete example of such a Lie category is the \textit{endomorphism bundle} $E^*\otimes E\rightarrow X$ of a vector bundle $E\rightarrow X$, with the composition given on simple tensors as $(\varphi_2\otimes v_2)(\varphi_1\otimes v_1)=\varphi_2(v_1)\,\varphi_1\otimes v_2$, and extended by bilinearity; this can easily be identified with the composition of linear maps $E_x\rightarrow E_x$, so $E^*\otimes E\subset \End(E)$ is a subcategory of $\End(E)$. Using Lemma \ref{lem:lie_subcat}, it is not hard to see that it is actually an embedded Lie subcategory of $\End(E)$, with respect to the inclusion map.

Another concrete example of a bundle of Lie monoids is the \textit{exterior bundle}, 
\[\Lambda(E)=\bigoplus_{i=0}^{\rank E}\Lambda^k(E),\] of an $\mathbb F$-vector bundle $E\rightarrow X$, i.e.\ $\Lambda(E)$ consists of all multivectors in $E$, and composition of $\alpha\in \Lambda^k(E_x)$ and $\beta\in \Lambda^l(E_x)$ is given as $\alpha\wedge\beta$. The units are given by $1_x=1\in \mathbb F=\Lambda^0(E_x)$, for any $x\in X$. Again, the given composition map is smooth, which follows easily from bilinearity of wedge product; moreover, $\Lambda(E)$ is again enriched over $\mathbf{Vect}$. In general, bundles of Lie monoids enriched over $\textbf{Vect}$ would rightfully be called \textit{smooth bundles of unital associative algebras}.
\end{ex}
\begin{ex}[Action categories]
An \textit{action} of a Lie monoid $M$ on a smooth manifold $X$ is a smooth map $\phi\colon M\times X\rightarrow X$, denoted by $\phi(g,x)=gx$, which satisfies $ex=x$ and $g(hx)=(gh)x$, for any $x\in X$ and $g,h\in M$. We observe that contrary to the case of Lie group actions, the map $\phi$ may not be submersive since the action is no longer by automorphisms of $X$, thus the target map in the naïve generalization of the action groupoid would not be a submersion.

	To remedy this, we construct the \textit{action category} of a given Lie monoid action $\phi\colon M\times X\rightarrow X$ as follows. Denote by
	$$
	M\ltimes X = \set{(g,x)\in M\times X\given \phi \text{ is a submersion at }(g,x)}
	$$
	the set of regular points of the action $\phi$.\footnote{Intuitively, the action category accounts only	 for the non-critical dynamics pertaining to the given action.} Defining the structure maps as usual, i.e.\ $s(g,x)=x$, $t(g,x)=gx$, the units as $u(x)=(e,x)$, and the composition as
	\[
	(g,hx)(h,x)=(gh,x),
	\] 
	we obtain a Lie category $M\ltimes X\rra X$. Indeed, since $M\ltimes X\subset M\times X$ is an open subset, we only need to check that if $\phi$ is a submersion at the points $(h,x)$ and $(g,hx)$, it is also a submersion at $(gh,x)$. To that end, we first notice that we may write the condition $g(hx)=(gh)x$ as an equality of maps $M\times M\times X\rightarrow X$,
	\[
	\phi\circ(\pr_1,\phi\circ(\pr_2,\pr_3))=\phi\circ(m\circ(\pr_1,\pr_2),\pr_3),
	\]
	where $m\colon M\times M\rightarrow M$ denotes the multiplication in the Lie monoid $M$. Differentiating this equality at $(g,h,x)$, we obtain
	\[
	\d\phi_{(g,hx)}\circ (\pr_1,\d\phi_{(g,x)}\circ(\pr_2,\pr_3))=\d\phi_{(gh,x)}\circ (\d m_{(g,h)}\circ (\pr_1,\pr_2),\pr_3).
	\]
	By assumption, the left-hand side is surjective, thus the same holds for the first map on the right-hand side. Hence $M\ltimes X\rra X$ is a category and thus a Lie category.
\end{ex}

\begin{ex}
\label{ex:order_cat}
A simple, yet important example of a Lie category is the \textit{order category} of $\R$, which is defined as the wide subcategory of the pair groupoid $\G=\R\times\R\rra \R$, given by
\[
\C=\set{(y,x)\in\R\times\R\given x\leq y}.
\]
The space of morphisms is thus the half-space below the diagonal in $\R^2$, and the  source and target maps are the projections to the vertical and the horizontal axis, respectively, implying that the boundary of $\C$ is regular. Moreover, since the inversion in the pair groupoid $\G$ is given by the reflection over the diagonal, the units in $\C$ are precisely the elements of the diagonal, and these are the only invertibles. 

\begin{figure}[h!]
	\includegraphics{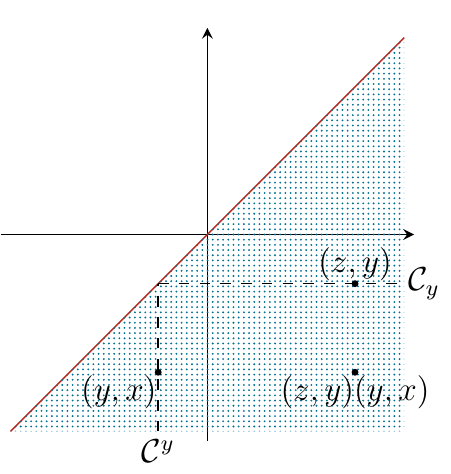}
\end{figure}
\noindent Although simple, this example has an important property: $\C$ can be seen as the preimage of the set $[0,\infty)$ under the functor 
\[
f\colon \R\times \R\rightarrow (\R,+), \quad f(y,x)=y-x.
\]
The following example is a generalization of this; we will make use of it when considering applications to statistical thermodynamics in Section \ref{sec:std}.
\end{ex}
\begin{ex}
\label{ex:preimage_subcat}
Suppose a Lie category $\D\rra \X$ without boundary is given, together with a smooth functor $f\colon \D\rightarrow (\R,+)$ such that:
\begin{samepage}
\begin{enumerate}
\item $0\in\R$ is a regular value of $f$.
\item $s|_{f^{-1}(0)}$ and $t|_{f^{-1}(0)}$ are submersions.
\end{enumerate}
\end{samepage}
Then the preimage $\C=f^{-1}\big([0,\infty)\big)$ is a Lie category. Indeed, since $0$ is a regular value of $f$, $\C$ is a smooth embedded submanifold in $\D$ with boundary $\partial\C=f^{-1}(0)$, see e.g.\ \cite[p.\ 62]{difftop}. Moreover, since $f^{-1}\big((0,\infty)\big)=\Int \C$ is open in $\D$, the restrictions $s|_\C,t|_\C$ of the source and target maps to $\C$ are submersions. Functoriality of $f$ implies that $\C$ is a wide subcategory of $\D$ with all invertibles contained within $\partial \C$, and moreover it also implies that $\partial \C$ is a wide subcategory of $\C$. The assumption (ii) enables us to use Lemma \ref{lem:lie_subcat} below to conclude that $\C$ is an embedded Lie subcategory of $\D$.
\end{ex}

\begin{lem}
\label{lem:lie_subcat}
Let $\C$ be a wide subcategory of a Lie category $\D\rra\X$. Suppose $\C$ is also an embedded submanifold of $\D$, such that $s|_\C, t|_\C$ are submersions and either of the following holds:
\begin{enumerate}
\item $\C$ has no boundary.
\item $\C$ has a boundary, and $s|_{\partial \C}, t|_{\partial\C}$ are submersions.
\end{enumerate}
Then $\C$ is an embedded Lie subcategory of $\D$.
\end{lem}
\begin{proof}
The only thing needed to be proven is that the restriction $m|_{\comp\C}\colon \comp \C\rightarrow \C$ of the composition map $m\colon \comp\D\rightarrow\D$, is smooth. To that end, it is enough to check that $\comp \C$ is a submanifold of $\comp \D$; notice that $\comp\C=\comp \D\cap (\C\times \C)$, so we will make use of the transversality theorem. 

Suppose first that $\C$ and $\D$ have no boundary, so that by the usual transversality theorem it is enough to check $\comp\D$ and $\C\times\C$ are transversal in $\D\times\D$, i.e.\
\[
T_{(g,h)}\comp \D+ T_g \C\oplus T_h\C=T_g\D\oplus T_h\D,\quad\text{for all }(g,h)\in \comp\C.
\]
To show this equality, let $(v,w)\in T_g\D\oplus T_h\D$. Since $s|_\C$ is a submersion, there is a vector $v'\in T_g\C$ with $\d s_g(v')=\d t_h(w)$. Define $v''=v-v'$, and now since $t|_\C$ is a submersion, there is a vector $w'\in T_h\C$ such that
\[
\d t_h(w')=\d t_h(w)-\d s_g(v'').
\]
Now define $w''=w-w'$. Clearly, $(v',w')\in T_g \C\oplus T_h\C$, and on the other hand, the definition of $w'$ and $w''$ imply
\[
\d s_g(v'')=\d t_h(w)- \d t_h(w')=\d t_h(w''),
\]
so that $(v'',w'')\in T_{(g,h)} \comp \D$, which concludes our proof for the boundaryless case. If $\D$ has a boundary, then the above proof together with Proposition \ref{prop:transversality_char} used on the inclusion $\C\times\C\xhookrightarrow{\iota} \D\times\D$, ensures that $\iota\pitchfork\comp\D$, and since $\comp\D\subset \D\times\D$ is a neat submanifold, $\comp\C\subset \C\times \C\subset \D\times\D$ is a submanifold by Proposition \ref{prop:transversality}, and hence since $\comp\C\subset \comp\D$, it is a submanifold of $\comp\D$.

If $\C$ has a boundary, then the assumption (ii) ensures that all the strata 
\[
\Int\C\times\Int\C, (\Int \C\times\partial \C)\cup(\partial \C\times\Int\C), \partial\C\times\partial\C
\]
of $\C\times\C$ are transversal to $\comp\D$ by a similar proof as above, so by Proposition \ref{prop:transversality_char} we again conclude $\iota\pitchfork \comp\D$.
\end{proof}
\begin{rem}
In particular, an embedded submonoid of a Lie monoid is its embedded Lie submonoid, however, a direct proof of this is much easier.
\end{rem}



\section{Reversibility of morphisms}
\textit{Reversible} (synonymously, \textit{invertible}) morphisms form an important subclass of morphisms in any category $\C$; recall that $g\in\C$ is said to be \textit{invertible}, if there is a unique morphism $g^{-1}\in \C$ such that $g^{-1}g=1_{s(g)}$ and $gg^{-1}=1_{t(g)}$. We observe that the set
\[
\G(\C)=\set{g\in\C\given g\text{ is invertible}}
\]
is a groupoid over the same base as $\C$. We call $\G(\C)$ the  \textit{core} of $\C$, the study of which begins with the following simple observation.
\begin{prop}
\label{prop:inv}
For any morphism $g$ in a category $\C$, the following are equivalent.
\begin{enumerate}[label={(\roman*)}]
\item $g$ is invertible.
\item Left and right translations by $g$ are bijections.
\item Left and right translations by $g$ are surjections.
\item $g$ has a left and a right inverse, i.e.\ there exist $g'\in \C^{s(g)}, g''\in \C_{t(g)}$ with $gg'=1_{t(g)}$ and $g''g=1_{s(g)}$. 
\end{enumerate}
\end{prop}
\begin{proof}
Implications $ (i) \Rightarrow (ii) \Rightarrow (iii)$ are clear, and $(iii)\Rightarrow (iv)$ follows by observing that $(iv)$ means $1_{t(g)}\in \Im(L_g)$, $1_{s(g)}\in \Im(R_g)$. Lastly, implication $(iv) \Rightarrow (i)$ follows from elementary abstract algebra: $g''=g''1_{t(g)}=g''gg'=1_{s(g)}g'=g'$.
\end{proof}
\begin{rem}
Note that injectivity of left and right translations corresponds to cancellative properties. For example, injectivity of $L_g$ is equivalent to stating that for any two $h,k\in \C^{s(g)}$, $gh=gk$ implies $h=k$.
\end{rem}

The regular boundary assumption on a Lie category $\C$ has the important consequence that the units dictate where invertible elements can be:
\begin{lem} 
\label{lem:regbdr}
Let $\C$ be a Lie category. For any invertible morphism $g\in \G(\C)$, the morphisms $g,g^{-1}, 1_{s(g)},1_{t(g)}$ must either all be contained in the interior $\Int\C$, or in the boundary $\partial\C$.
\end{lem}
\begin{proof}[Proof]
If $\C$ has no boundary then the lemma holds trivially, so suppose $\partial \C\neq \emptyset$. Invertibility of $g$ means $L_g\colon \C^{s(g)}\rightarrow \C^{t(g)}$ is a diffeomorphism, which maps $1_{s(g)}\mapsto g$, so we must have either $1_{s(g)}\in \partial(\C^{s(g)})$ and $g\in \partial(\C^{t(g)})$, or $1_{s(g)}\in \Int(\C^{s(g)})$ and $g\in \Int(\C^{t(g)})$. Regularity of boundary implies $\partial(\C^x)=\C^x\cap \partial \C$ for any $x\in \X$, so we must have either $1_{s(g)}\in \partial\C$ and $g\in \partial\C$, or $1_{s(g)}\in \Int\C$ and $g\in \Int\C$. A similar result is obtained for $1_{t(g)}$ and $g$ using the right translation $R_g$, and similarly for $g^{-1}$ and the units using $L_{g^{-1}}, R_{g^{-1}}$ since $s(g^{-1})=t(g)$.
\end{proof}

The last lemma implies that any Lie groupoid (a Lie category with all morphisms invertible) must have an empty boundary. We also obtain the following two immediate corollaries.

\begin{cor}
\label{cor:regbdr_inv}
For any Lie category $\C\rra \X$, there holds:
\begin{align}
u(\X)\subset \Int\C &\text{ implies } \G(\C)\subset \Int\C,\\
u(\X)\subset \partial\C &\text{ implies }\G(\C)\subset \partial\C.
\end{align}
\end{cor}
\begin{cor}
The invertible elements of any Lie monoid are either contained in its  interior, or in its boundary.
\end{cor}

An impending question is whether the core $\G(\C)$ of $\C$ is a Lie groupoid. Without additional assumptions this is false:
\begin{ex}
A simple example of a Lie category $\C$ with a regular boundary, whose core $\G(\C)$ is not a manifold, is the disjoint union of the order category on $\R$ and the pair groupoid on $\R$. More concretely, we define 
\begin{align*}
\C&=(-\infty,0)^2\cup \set{(y,x)\in (0,\infty)^2\given x\leq y},\\
\X&=\R\setminus \set 0,
\end{align*}
with the categorical structure induced by the pair groupoid structure on $\R$. 
\begin{figure}[h!]
\includegraphics{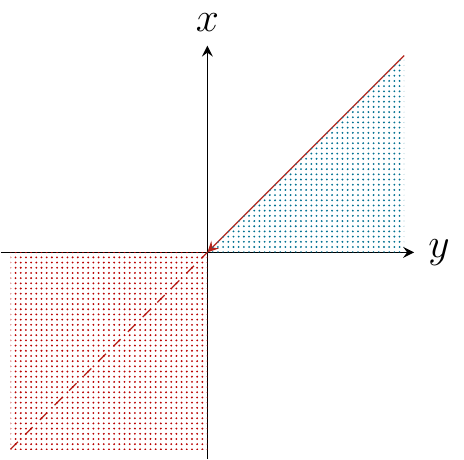}
\end{figure}

\noindent In this case, $\G(\C)=(-\infty,0)^2\cup \set{(x,x)\given x>0}$, which has two components of different dimensions. We have depicted this above, with invertible elements of $\C$ in red, and non-invertible in blue; the units in $\Int \C$ are depicted with a dashed line.

The culprit in the example above is that the units were allowed to be both in the interior and in the boundary.
\end{ex}

\begin{samepage}
\begin{defn}
\label{def:normal_bdry}
A Lie category $\C\rra\X$ is said to have a \textit{normal boundary}, if either of the following holds:\
\begin{enumerate}[label={(\roman*)}]
\item $u(\X)\subset \Int \C$.
\item $\partial\C$ is a wide subcategory of $\C$, i.e.\ $\partial \C$ is a subcategory of $\C$ and $u(\X)\subset\partial\C$.
\end{enumerate}
\end{defn}
\end{samepage}
\begin{rem}
A Lie category without a boundary trivially has a normal boundary. 
\end{rem}

The following result was proved by Charles Ehresmann in his pioneering paper \cite{ehr1959}, where he introduced Lie categories. Ehresmann proved it by implicitly assuming $\partial\C=\emptyset$, and with a somewhat old-school method of local coordinates; the proof we present is coordinate-free and uses the modern language of differentiable manifolds, is conceptually clearer, and holds for categories with a normal boundary.
\begin{thm}
\label{thm:open_inv}
If $\C$ is a Lie category with a normal boundary, then its core $\G(\C)$ is an embedded Lie subcategory of $\C$. More precisely, if $u(\X)\subset \Int\C$, then $\G(\C)$ is open in $\Int \C$, and if $u(\X)\subset \partial \C$, then $\G(\C)$ is open in $\partial \C$.
\end{thm}
\begin{proof}
Consider first the case when $u(\X)\subset \Int\C$. 
We need to show that any $k\in \G(\C)$ has an invertible neighborhood in $\C$; we will do so by showing that $k$ admits a left-invertible and a right-invertible neighborhood. With this motive, define the map
\begin{align}
\vartheta\colon \comp\C\rightarrow {\C\tensor[_t]{\times}{_t}}\,\C,\quad \vartheta(g,h)=(g,gh),
\end{align}
so $\vartheta=(\pr_1,m)$. We claim that this is a local diffeomorphism at the point $(k,k^{-1})$. Note that by equation \ref{eq:regbdr_c2} and a similar result for $\partial({\C\tensor[_t]{\times}{_t}}\,\C)$, we have
\[
(k,k^{-1})\in\Int(\comp\C)\quad\text{ and }\quad(k,1_{t(k)})=\vartheta(k,k^{-1})\in \Int({\C\tensor[_t]{\times}{_t}}\,\C),
\]
so that it is enough to show $\d\vartheta_{(k,k^{-1})}$ is an isomorphism, by virtue of the inverse map theorem (for boundaryless manifolds). For dimensional reasons, it is enough to check $\vartheta$ is an immersion at $(k,k^{-1})$.

Suppose $\d\vartheta_{(k,k^{-1})}(v,w)=0$ for some $(v,w)\in T_{(k,k^{-1})}(\comp\C)$. The identity 
$
\d \vartheta(v,w)=(v,\d m(v,w))
$ implies $v=0$, so we obtain
\begin{align*}
\d m_{(k,k^{-1})}(v,w)=\d (L_k)_{k^{-1}}(w).
\end{align*}
Since $L_k$ is a diffeomorphism, we conclude $w=0$, hence $\vartheta$ is a local diffeomorphism at $(k,k^{-1})$, i.e.\ there is a neighborhood $U$ of $(k,k^{-1})$ which is mapped diffeomorphically onto a neighborhood $V$ of $(g,1_{t(g)})$.

Now note that $\C$ embeds into ${\C\tensor[_t]{\times}{_t}}\,\C$, by the map
$
\lambda(g)=(g,1_{t(g)}),
$
so the set $\vartheta^{-1}(V\cap\lambda(\C))$ is diffeomorphic to $V\cap \lambda(\C)$, and consists of pairs $(g,g')\in U$ such that $gg'=1_{t(g)}$. In other words, $\lambda^{-1}(V)$ is a neighborhood of $k$, elements of which have right inverses. 

Similarly, we can show that the map $\tilde\vartheta\colon \comp\C\rightarrow {\C\tensor[_s]{\times}{_s}}\,\C, (g,h)\mapsto(gh,h)$ is a local diffeomorphism at $(k^{-1},k)$, and use the embedding $\tilde\lambda\colon \C\rightarrow {\C\tensor[_s]{\times}{_s}}\,\C$, $\tilde\lambda(g) = (1_{s(g)},g)$ to obtain a neighborhood $\tilde\lambda^{-1}(\tilde V)$ of $k$, elements of which have left inverses. To conclude, note that Proposition \ref{prop:inv} guarantees $\lambda^{-1}(V)\cap \tilde\lambda^{-1}(\tilde V)$ is an invertible neighborhood of $k$.

Finally, for the case $u(\X)\subset \partial\C$, just note that normality of the boundary can be used with Lemma \ref{lem:lie_subcat} to conclude that $\partial\C$ is an embedded Lie subcategory of $\C$ without boundary, so we can apply the previous case to $\partial \C$.
\end{proof}

Given a morphism $F\colon \C\rightarrow \D$ between Lie categories with normal boundaries, we have $F(\G(\C))\subset \G(\D)$ by functoriality, so we may define $\G(F)=F|_{\G(\C)}$. The map $\G\colon \mathbf{LieCat_\partial}\rightarrow\mathbf{LieGrpd}$ defines a functor from the category of Lie categories with normal boundary to the category of Lie groupoids without boundary, which is easily seen to be right adjoint to the inclusion functor $\mathbf{LieGrpd}\hookrightarrow \mathbf{LieCat_\partial}$. 

Using the last theorem we can also show that the universal property of the core extends to the differentiable setting. We note that more consequences of the Theorem \ref{thm:open_inv} will be explored in upcoming sections.
\begin{cor}
Let $\C$ be a Lie category with a normal boundary, and let $\mathcal H\subset \C$ be a groupoid that is also a Lie subcategory of $\C$. If any smooth morphism $F\colon \G\rightarrow \C$, defined on a Lie groupoid $\G$, factors uniquely through $\mathcal H$, then $\mathcal H=\G(\C)$. 
\end{cor}
\begin{proof}
The inclusion $\mathcal H\subset \G(\C)$ is an easy consequence of functoriality of $\mathcal H\hookrightarrow \C$, since functors map isomorphisms to isomorphisms. For the converse inclusion, we pick the smooth map $F\colon \G(\C)\hookrightarrow \C$ and now the existence of $\bar F\colon \G(\C)\rightarrow \mathcal H$ such that $F=\iota\circ \bar F$, ensures that if  $g\in \G(\C)$, then $g=\bar F(g)\in \mathcal H$.
\end{proof}

\section{Lie algebroids of a Lie category}
The construction of a Lie algebroid using left-invariant (or right-invariant) vector fields on a Lie groupoid readily generalizes to Lie categories; however, we no longer have a canonical isomorphism between the two algebroids, previously given by inversion map. Below, we present the main idea of the construction (\ref{defn:left_inv} through \ref{prop:lie_alg}), mainly to fix notation; we omit some details which can be found in standard Lie groupoid references.

\begin{defn}
\label{defn:left_inv}
A \textit{left-invariant} vector field on a Lie category $\C\rra \X$ is a vector field $X\in\vf(\C)$, which is tangent to $t$-fibres, i.e.\ $X\in\Gamma^\infty(\ker \d t)$, and left-invariant, i.e.\ $\d(L_g)_h(X_h)=X_{gh}$ for all $(g,h)\in \comp\C$. Denote by $\vf^L(\C)$ the vector space of left-invariant vector fields on $\C$.
\end{defn}

\begin{lem}
\label{lem:closed_lie_bracket}
Let $\C$ be a Lie category. The vector space $\vf^L(\C)$ is closed under the Lie bracket, and canonically isomorphic to the vector space $\Gamma^\infty(A^L(\C))$ of sections of the vector bundle $A^L(\C)=u^*(\ker\d t)$ over $\X$.
\end{lem}
\begin{proof}
Closedness under the Lie bracket is a consequence of the fact that if $X,Y$ are $L_g$-related to themselves when restricted to appropriate fibres, so is $[X,Y]$. The canonical isomorphism $\mathrm{ev}\colon\vf^L(\C)\rightarrow \Gamma^\infty(A^L(\C))$ is given by restriction to the units, and its inverse $\mathrm{ev}^{-1}$ maps any section $\alpha\in \Gamma^\infty(A^L(\C))$ to its left-invariant extension $\alpha^L$, given as
$
\alpha^L(g)=\d(L_g)_{1_{s(g)}}(\alpha_{s(g)}),
$
which is a smooth section $\C\rightarrow T\C$ since it can be realized as the composition $\alpha^L=\d m\circ \tau^\alpha$, where $\tau^\alpha\colon \C\rightarrow T(\comp\C)$ is given by $g\mapsto (0_g,\alpha_{1_{s(g)}}).$
\end{proof}
\begin{defn}
The \textit{left Lie algebroid} of a Lie category $\C\rra \X$ is the vector bundle $A^L(\C)\rightarrow \X$, endowed with the Lie bracket $[\cdot,\cdot]$ on its sections as induced by the isomorphism $\mathrm{ev}$, together with the \textit{anchor map} $\rho^L\colon A^L(\C)\rightarrow T\X$, $\rho^L=\d s|_{A^L(\C)}$.
\end{defn}
Similarly, we define the \textit{right Lie algebroid} of a Lie category $\C$ as the vector bundle $A^R(\C)=u^*(\ker \d s)$, which as a vector space is isomorphic to the space $\vf^R(\C)$ of right-invariant vector fields, and its anchor map is defined as $\rho^R=\d t|_{A^R(\C)}$.

The proof of the next proposition is again the same as with the Lie groupoid case, so we omit it. Again, an analogous result holds for the right Lie algebroid.
\begin{prop}
\label{prop:lie_alg}
The left Lie algebroid $A^L(\C)$ of a Lie category $\C\rra\X$ satisfies:
\begin{enumerate}[label={(\roman*)}]
\item $\rho^L$ preserves the brackets, i.e.\ $\rho^L[\alpha,\beta]=[\rho^L\alpha,\rho^L\beta]_{T\X}$,
\item Leibniz rule: $[\alpha,f\beta]=f[\alpha,\beta]+\rho^L(\alpha)(f)\beta$,
\end{enumerate}
for all $\alpha,\beta\in \Gamma^\infty(A^L(\C))$ and $f\in C^\infty(\X).$
\end{prop}

Any morphism $F\colon \C\rightarrow \D$ over $\id_\X$ of Lie categories over the same object manifold $\X$, induces morphisms between their left and right Lie algebroids (respectively), denoted by
\begin{align*}
F_*^L\colon A^L(\C)\rightarrow A^L(\D)\quad\text{and}\quad F_*^R\colon A^R(\C)\rightarrow A^R(\D),
\end{align*}
and defined on respective sections of $A^L(\C)$ and $A^R(\C)$ in the obvious way:
\begin{align*}
\alpha \mapsto \d F\circ \alpha.
\end{align*}
In the case when the object manifolds of $\C$ and $\D$ are not equal and the morphism $F$ does not restrict to a diffeomorphism between the units, we encounter the same complications as in the Lie groupoid case; this is resolved in the same manner as for Lie groupoids, and since we will not be needing this more general result, we point the reader to \cite[Chapter 4.3]{mackenzie} for details. The upshot is that $A^L$ and $A^R$ are functors from the category $\mathbf{LieCat}$ of Lie categories to the category $\mathbf{LieAlgd}$ of Lie algebroids.

As mentioned, we do not have a canonical isomorphism between the left and right Lie algebroid of a Lie category, which is given in a Lie groupoid by the inversion map. However, we have the following result, and we will later encounter a related one when studying extensions of categories to groupoids (see section \ref{sec:ext}).
\begin{prop}
Let $\C\rra \X$ be a Lie category. If the units of $\C$ are contained in the interior of $\C$, i.e.\ $u(\X)\subset \Int\C$, then the left and right Lie algebroids of $\C$ are isomorphic to the Lie algebroid of its core $\G(\C)$.
\end{prop}
\begin{proof}
By Theorem \ref{thm:open_inv}, $\G(\C)$ is open in $\C$, so we get the following chain of isomorphisms of Lie algebroids:
\begin{align}
\label{eq:chain_iso}
A^L(\C)\cong A^L(\G(\C))\cong A^R(\G(\C))\cong A^R(\C),
\end{align}
where the first and last isomorphism are induced by the inclusion $\G(\C)\hookrightarrow \C$, and the isomorphism in the middle is induced by inversion in the groupoid $\G(\C)$.
\end{proof}

\begin{rem}
Note that if $\C$ has a normal boundary and $u(\X)\subset \partial\C$, the Lie algebroid $A(\G(\C))$ of the core will always fail to be isomorphic to the two Lie algebroids of $\C$, since the rank of the vector bundle $A(\G(\C))$ is one less than the rank of $A^L(\C)$ and $A^R(\C)$. This is demonstrated by the following two examples: 
\begin{enumerate}[label={(\roman*)}]
\item The two Lie algebras of the Lie monoid $M=\mathbb H^n$ are isomorphic (as vector spaces) to $A_L(M)\cong \R^n\cong A_R(M)$, whereas $A(G(M))\cong\R^{n-1}$ since the core of $M$ is $G(M)=\R^{n-1}$.
\item Consider the order category $\C=\set{(y,x)\in \R^2\given x\leq y}\rra \R$ from Example \ref{ex:order_cat}. Notice that its core $\G(\C)$ is just the base groupoid over $\R$, hence its Lie algebroid is the zero bundle $A(\G(\C))=\R\times \set 0$. On the other hand, the left and right Lie algebroid of $\C$ are both isomorphic to $T\R$.
\end{enumerate}
\end{rem}

\section{Ranks of morphisms}
What follows can be seen as a natural generalization of the usual notion of rank from linear algebra (see Example \ref{ex:ranks} (iii)).
\begin{defn}
\label{defn:rank}
Let $\C\rra\X$ be a Lie category and let $g\in\C$. The \textit{left rank} and \textit{right rank} of $g$ are defined as
\begin{align*}
\rankl(g)&=\rank\d(L_g)_{1_{s(g)}},\\
\rankr(g)&=\rank\d(R_g)_{1_{t(g)}}.
\end{align*}
If the left and right rank of a morphism $g$ are equal, we just write $\rank(g)=\rankl(g)=\rankr(g)$ and call this integer the \textit{rank} of $g$.
Moreover, we say $g$ has \textit{full rank}, if its left and right ranks are full, that is, if
\[
\rank(g)=\codim_\C(\X)=\dim\C-\dim \X=:\delta.
\]
In this case, we will sometimes say that $g$ is a \textit{regular} morphism. If $g$ is not regular, we will call it \textit{singular}. 

Finally, we say that $g$ has \textit{constant} left rank, if $\rankl(g)=\rank\d(L_g)_h$ for all $h\in \C^{s(g)}$, and similarly that $g$ has \textit{constant} right rank, if $\rankr(g)=\rank\d(R_g)_h$ for all $h\in \C_{t(g)}$. To avoid ambiguity, we will sometimes write $\rank_\C$ instead of $\mathrm{rank}.$
\end{defn} 
\begin{ex}\
\begin{enumerate}[label={(\roman*)}]
\label{ex:ranks}
\item All invertible morphisms in any Lie category have full and constant rank.
\item In a Lie monoid $M$, the ranks of an element $g\in M$ are just the ranks of $\d(L_g)_e$ and $\d(R_g)_e$. If $M$ is an abelian Lie monoid, then clearly any element of $M$ has equal left and right rank. The example $M=\mathbb H^n$ shows that regularity does not imply invertibility. 

We will later see that regularity and constancy of ranks of all morphisms in a Lie category is ensured whenever dealing with a Lie subcategory of a Lie groupoid, as is with the simple example $\mathbb H^n\hookrightarrow \R^n$ for addition.
\item If $M=\R^{n\times n}$, let us show how the above notion of rank relates with the usual one. If $A\in\R^{n\times n}$, then the usual notion reads $\rank A=\dim \Im A$, when $A$ is seen as the map $\R^n\rightarrow \R^n$. On the other hand, the left rank in Definition \ref{defn:rank} equals:
\[\rankl_M(A)=\dim \Im \d(L_A)_I=\dim \Im L_A,\]
where $L_A\colon \R^{n\times n}\rightarrow \R^{n\times n}$ is the left translation by $A$, and the last equality follows by linearity. Similarly, $\rankr_M(A)=\dim\Im R_A$. Denoting by $E_{ij}$ the matrix with 1 in place $(i,j)$ and zero elsewhere, we have that $AE_{ij}$ has $i$-th column of $A$ in $j$-th column and is zero elsewhere, and $E_{ij}A$ has $j$-th row of $A$ in $i$-th row and is zero elsewhere, so:
\begin{align*}
\Im(L_A)&=\operatorname{Lin}(AE_{ij})_{i,j=1}^n=\set[\Big]{\begin{bmatrix}v_1\dots v_n\end{bmatrix}\given v_i\in \Im(A)},\\
\Im(R_A)&=\operatorname{Lin}(E_{ij}A)_{i,j=1}^n=\set*{\begin{bmatrix}v_1\dots v_n\end{bmatrix}^{\mathsf T}\given v_i\in \Im\left(A^{\mathsf T}\right)}.
\end{align*}
Since $\rank A=\rank A^{\mathsf T}$, it follows that 
\[
\rankl_M(A)=\rankr_M(A)=n\rank(A),
\]
and we see that $A$ is regular if, and only if, $A$ is invertible. 

The above result readily generalizes to arbitrary finite-dimensional vector spaces: if $V$ is a vector space, then
$\rank_{\End(V)}(A)=\dim (V)\rank (A).$ 
Moreover, for any vector bundle $E$, the above result clearly also generalizes to the endomorphism bundle $E^*\otimes E$, so that for any $A\colon E_x\rightarrow E_x$,
\[
\rank_{E^*\otimes E}A=\rank (E)\rank (A).
\]
\item Examples of Lie monoids that do not have coinciding left and right ranks may be found in the context of finite-dimensional unital algebras. As a concrete example, take the algebra $A\subset \R^{2\times 2}$ of upper-diagonal real $2\times 2$ matrices, the canonical basis of which is given by $a = \bigl[ \begin{smallmatrix}1 & 0\\ 0 & 0\end{smallmatrix}\bigr], b = \bigl[ \begin{smallmatrix}0 & 0\\ 0 & 1\end{smallmatrix}\bigr], c = \bigl[ \begin{smallmatrix}0 & 1\\ 0 & 0\end{smallmatrix}\bigr]$. Identifying $T_e A$ with $A$ and noting that left and right translations are linear maps, we can identify $\d(L_g)_e = L_g$ for any $g\in A$, and similarly for the right translations. It is easy to see $\Im(R_a)=\operatorname{Lin}(a)$ and $\Im(L_a)=\operatorname{Lin}(a,c)$ by computing all the products of $a$ with the canonical basis, so we conclude that $2=\rankl_A(a)\neq\rankr_A(a)=1$.
\end{enumerate}
\end{ex}

\begin{prop}[Properties of ranks]
In a Lie category $\C\rra \X$, there holds:
\begin{enumerate}[label={(\roman*)}]
\item The ranks of a composition of composable morphisms $g,h\in\C$ are bounded from above:
\begin{align}
\rankl(gh)&\leq \rankl(h),\\
\rankr(gh)&\leq \rankr(g).
\end{align}
\item The ranks of $g\in \C$ are bounded from below by the ranks of anchors $\rho^L\colon A^L(\C)\rightarrow T\X$ and $\rho^R\colon A^R(\C)\rightarrow T\X$ of Lie algebroids of $\C$:
\begin{align}
\rankl(g)\geq \rank\rho_{s(g)}^L,\\
\rankr(g)\geq \rank\rho_{t(g)}^R.
\end{align}
\end{enumerate}
\end{prop}
\begin{proof}
Since $L_{gh}=L_g\circ L_h$, we have
\[
\rankl(gh)=\rank(\d(L_g)_h\circ\d(L_h)_{1_{s(h)}})\leq \rankl(h),
\]
and similarly for the right translation by using $R_{gh}=R_h\circ R_g$, so $(i)$ follows. For the property $(ii)$ note that the diagram
\[
\begin{tikzcd}[column sep=small]
\C^{s(g)} \arrow[rr, "L_g"] \arrow[rd, "s|_{\C^{s(g)}}"'] &    & \C^{t(g)} \arrow[ld, "s|_{\C^{t(g)}}"] \\
& \X &
\end{tikzcd}
\]
implies $\rankl(g)\geq \rank \d(s|_{\C^{s(g)}})_{1_{s(g)}}=\rank \rho^L_{s(g)}$.
\end{proof}

In the context of Lie groupoids, it is well known that the composition map is a submersion. This is no longer true for Lie categories in general; a counterexample is provided by the Lie monoid $\R^{n\times n}$ for matrix multiplication, where $\d m_{(0,0)}$ is easily seen to be the zero map. However, the following results ensure that $m$ is a submersion in case all morphisms have full and constant rank; as mentioned, this is the case for Lie categories extendable to Lie groupoids, as we will see later in Lemma \ref{lem:extensions}.

\begin{lem}
\label{lem:theta}
Let $\C \rra \X$ be a Lie category and $g\in\C$. The map $L_g$ has full rank at $h\in \C^{s(g)}$ if, and only if, the map 
$\vartheta\colon \comp\C\rightarrow {\C\tensor[_t]{\times}{_t}}\,\C, (g,h)\mapsto(g,gh)$
has full rank at $(g,h)$. In particular, $g$ has full left rank if, and only if, $\vartheta$ has full rank at $(g,1_{s(g)})$.
\end{lem}
Recall that we have already encountered the map $\vartheta$ in the proof of Theorem \ref{thm:open_inv}. In a Lie groupoid, this map is a bijection with inverse $(g,h)\mapsto (g,g^{-1}h)$, so above lemma guarantees it is a diffeomorphism, which can be used e.g.\ to show that the inversion map is automatically smooth, by realizing it as the composition \[g\mapsto (g,1_{t(g)})\xmapsto{\vartheta^{-1}} (g,g^{-1})\mapsto g^{-1}.\]

An analogous result to Lemma \ref{lem:theta} of course holds for ranks of right translations, using the map $\tilde\vartheta\colon \comp\C\rightarrow {\C\tensor[_s]{\times}{_s}}\,\C$, given by $(g,h)\mapsto (gh,h)$. For instance, the second part of above lemma would then read: $g$ has full right rank if, and only if, $\tilde\vartheta$ has full rank at $(1_{t(g)},g)$. 
\begin{proof}
For the forward implication, note that $\d\vartheta_{(g,h)}(v,w)=(v,\d m_{(g,h)}(v,w))$ holds for all $(v,w)\in T_{(g,h)}\comp\C$, hence $\d\vartheta_{(g,h)}(v,w)=(0,0)$ first implies $v=0$, and since $\d m_{(g,h)}(0,w)=\d (L_g)_h(w)$, we get $w=0$ since $L_g$ has full rank at $h$. For the other direction, note that $\d(L_g)_h(w)=0$ implies $\d\vartheta_{(g,h)}(0,w)=(0,0)$, so the assumption that $\vartheta$ has full rank at $(g,h)$ yields $w=0$.
\end{proof}
\begin{cor}
\label{cor:composition_submersion}
Let $\C\rra \X$ be a Lie category and $g\in\C$. If the map $L_g$ has full rank at $h\in \C^{s(g)}$, then the composition map $m\colon\comp\C\rightarrow \C$ is a submersion at $(g,h)$. Hence if all morphisms have full and constant rank, the composition is submersive.
\end{cor}
\begin{proof}
Differentiating the identity $m=\pr_2\circ\vartheta$  yields 
$
\d m_{(g,h)}=\d(\pr_2)_{(g,gh)}\circ\d\vartheta_{(g,h)},
$
which is a composition of surjective maps by previous lemma and the fact that $\pr_2\colon {\C\tensor[_t]{\times}{_t}}\,\C\rightarrow\C$ is a submersion.
\end{proof}

We now direct our attention to the subsets of $\C$ of morphisms with full rank. Denote by 
\begin{align*}
\rk_r^L(\C)&=\set{g\in \C\given \rankl(g)=r},\\
\rk_r^R(\C)&=\set{g\in \C\given \rankr(g)=r}
\end{align*}
the sets of morphisms with left and right rank equal to $r$, respectively, and furthermore by $\rk_r(\C)=\rk_r^R(\C)\cap \rk_r^L(\C)$ morphisms whose both ranks equal $r$.
\begin{prop}
In any Lie category $\C\rra\X$, the subset $\rk_\delta(\C)\subset\C$ of regular morphisms is open. 
\end{prop}
\begin{proof}
Let us first show that differentials of left translations define a certain morphism of vector bundles. We consider the following vector bundle over $\C$:
\[E^L=\coprod_{g\in \C}(\ker\d t_{1_{s(g)}})^*\otimes\ker\d t_g=(s^*u^*\ker \d t)^*\otimes  \ker \d t.\]
Denote the projection map by $p_L\colon E^L\rightarrow \C$, so the fibre $p_L^{-1}(g)=E^L_g$ of $E^L$ consists of all linear maps $\ker \d t_{1_{s(g)}}\rightarrow \ker\d t_g$.  Note that the map $\phi_L\colon \C\rightarrow E^L$, given as $g\mapsto \d(L_g)_{1_{s(g)}}\colon \ker \d t_{1_{s(g)}}\rightarrow \ker\d t_g$ is a section of $E^L$, which is smooth since 
\[
\d (L_g)_{1_{s(g)}}=\d m_{(g,1_{s(g)})}(0_g,-).
\]
In other words, $\phi_L$ defines a morphism $s^*u^*\ker\d t\rightarrow \ker \d t$ of vector bundles. 

Now take an atlas of local trivializations $\psi_i\colon p_L^{-1}(U_i)\rightarrow U_i\times \R^{\delta\times\delta}$ of $E^L$. The set $U_i\cap \rk^L_\delta(\mathcal \C)=\phi_L^{-1}\circ\psi_i^{-1}(U_i\times \mathrm{GL}(\delta,\R))$ is open in $U_i$, hence
\[
\rk_\delta^L(\C)=\bigcup_i \big(U_i\cap \rk_\delta^L(\C)\big)
\]
is open in $\C$. A similar proof works for $\rk^R_\delta(\C)$, so the result for $\rk_\delta(\C)$ follows.
\end{proof}

\begin{rem}
The smooth map $\phi_L$ (respectively, $\phi_R$) which was used in the last proposition, can also easily be used to show that $\rankl$ (respectively, $\rankr$) is a lower semi-continuous function, i.e.\ that any morphism $g\in\C$ admits a neighborhood on which the left (respectively, right) rank is non-decreasing. Note that this should not be confused with lower semi-continuity of the map $\C^{s(g)}\rightarrow \R$, given as $h\mapsto \rank \d(L_g)_h$, where $g\in \C$ is a fixed morphism.
\end{rem}

Properties of ranks of morphisms in a Lie category $\C$ reflect on its categorical structure, as illustrated by the following simple observation.

\begin{cor}
Let $\C$ be a Lie category. If the regular morphisms have constant rank, then $\rk_\delta(\C)$ is an open Lie subcategory of $\C$.
\end{cor}
\begin{proof}
We only have to check that the composition of two regular morphisms is a regular morphism. To that end, just notice that
$
\d(L_{gh})_{1_{s(h)}}=\d(L_g)_h\circ\d(L_h)_{1_{s(h)}}
$
is a composition of maps with full rank, and similar holds for right translations.
\end{proof}

\begin{rem}
Notice that in this case, the left and right algebroids of $\rk_\delta(\C)$ are isomorphic to those of $\C$, respectively. This is in accord with the moral that the left and right algebroids are, as vector bundles, determined by their fibres at regular morphisms.
\end{rem}


\subsection*{Action of the core $\G(\C)$ on a Lie category $\C$}

The natural left and right actions of the core $\G(\C)$ on a Lie category $\C$ can be used to find properties of Lie categories.\footnote{A reference for actions of Lie groupoids on smooth maps is \cite[Chapter 1.6]{mackenzie}.} We will only focus on describing the left action of $\G(\C)$ on $t\colon \C\rightarrow \X$; the right action of $\G(\C)$ on $s\colon \C\rightarrow \X$ follows a similar construction. 

Denote by $\G(\C) \tensor[_s]{\times}{_t}\C=\set{(g,c)\in \G(\C)\times\C\given s(g)=t(c)}$ the fibred product of the maps $s|_{\G(\C)}$ and $t$. This set has a natural structure of a groupoid over $\C$; indeed, the source and target maps are given as $\underline s(g,c)=c$ and $\underline t(g,c)=gc$, the composition is given as $(g,hc)(h,c)=(gh,c)$, the unit map is $\underline u(c)=(1_{t(c)},c)$ and the inverses are given by $(g,c)^{-1}=(g^{-1},gc)$. We leave it to the reader to check that this defines a groupoid over $\C$. As regards to its smooth structure:

\begin{prop}
Let $\C\rra\X$ be a Lie category without boundary. The left action of the core $\G(\C)$ on the target map $t\colon\C\rightarrow \X$ yields a Lie groupoid
\[
\G(\C)\tensor[_s]{\times}{_t}\C\rra\C.
\]
\end{prop}
\begin{proof}
The core $\G(\C)$ is an open Lie subgroupoid of $\C$ by Theorem \ref{thm:open_inv}, and so the usual transversality theorem for manifolds without boundary implies that 
\[\underline\C=\G(\C)\tensor[_s]{\times}{_t}\C=(s|_{\G(\C)}\times t)^{-1}(\Delta_\X)\]
is a smooth submanifold of $\G(\C)\times\C$ without boundary. Furthermore, $\underline s=\mathrm{pr}_2$ is clearly a sumbersion, and Corollary \ref{cor:composition_submersion} tells us that the target map \[\underline t=m|_{\G(\C)\tensor[_s]{\times}{_t}\C}\] is also a submersion, thus $\underline\C{}^{(2)}\subset \underline\C\times\underline\C$ is also a submanifold by transversality. The smoothness of composition map $\underline m$ and unit map $\underline u$ then follows easily from the smoothness of respective maps in $\C$.
\end{proof}
\begin{rem}
\label{rem:base_boundary}
Notice that in the case when $\C$ has a normal boundary, the maps $\underline s$ and $\underline t$ are still submersions, but their restrictions to the boundary $\partial\underline\C=\G(\C)\times \partial\C$ are not, which shows the need for amending the definition of a Lie category when the object manifold has a boundary (or corners). 

In this case, $\underline \C\subset \G(\C)\times \C$ is still a submanifold (by Proposition \ref{prop:transversality}) and the maps $\underline s$ and $\underline t$ are topological submersions (see Definition \ref{defn:top_sub}) that also satisfy $\underline s(\partial\underline\C)\subset \partial\C$ and $\underline t(\partial\underline\C)\subset \partial\C$, so \cite[Proposition 4.2.1]{corners} implies that $\underline s$-fibres and $\underline t$-fibres are submanifolds of $\underline \C$. This suggests a definition of a Lie category for the case when the object manifold has corners, but we will not pursue this further here.
\end{rem}
We can now use the action Lie groupoids above to prove that it does not matter at which invertible morphism we measure the ranks of a given morphism $g\in\C$.
\begin{cor}
\label{cor:ranks_invertibles}
Let $\C\rra\X$ be a Lie category with a normal boundary. For any $g\in\C$, there holds:
\begin{align*}
\rankl(g)&=\rank \d(L_g)_h\ \text{ for any }h\in \G(\C)^{s(g)},\\
\rankr(g)&=\rank \d(R_g)_h\ \text{ for any }h\in \G(\C)_{t(g)}.
\end{align*}
\end{cor}
\begin{proof}
Consider first the case when $\partial\C=\emptyset$. The restriction of the target map to any source fibre has constant rank, in any Lie groupoid.\footnote{See proof of Corollary \ref{cor:hom_sets_embedded} for a method of proving this.} In our case, $\underline s^{-1}(g)=\G(\C)_{t(g)}$ for any $g\in\C$, and the map
\[\underline t|_{\underline s^{-1}(g)}=R_g|_{\G(\C)_{t(g)}}\]
has constant rank which must thus be equal to $\rankr(g)$. A similar proof works for left translations, using the right action of $\G(\C)$ on $s\colon\C\rightarrow \X$.

Now suppose $\partial\C\neq \emptyset$. We may consider the left action groupoids 
\[
\G(\C)\tensor[_s]{\times}{_t}\Int\C\rra\Int\C,\quad \G(\C)\tensor[_s]{\times}{_t}\partial\C\rra\partial\C
\]
of $\G(\C)$ on $t|_{\Int\C}$ and $t|_{\partial \C}$, respectively; note that these are in fact actions of $\G(\C)$ since $L_g(\Int \C^{s(g)})\subset \Int \C^{t(g)}$ and $L_g(\partial\C^{s(g)})\subset\partial\C^{t(g)}$ for any $g\in\G(\C)$. The same technique as above now shows the wanted conclusion.
\end{proof}

\subsection*{Singular distributions associated to translations}

A Lie category $\C\rra\X$ (assume it is boundaryless for simplicity) comes equipped with singular distributions determined by differentials of left and right translations. Focusing only on left translations, define the singular distribution $D\subset T\C$ as
\[D_g=\Im\d(L_g)_{1_{s(g)}}\leq \ker\d t_g\leq T_g\C.\]
In what follows, we prove that $D$ is \textit{integrable}, i.e.\ there is a decomposition $\mathcal F(D)$ of $\C$ into maximally connected weakly embedded submanifolds,\footnote{Recall that an injective immersion $\varphi\colon M \rightarrow N$ is said to be a \textit{weak embedding}, if any smooth map $f\colon P\rightarrow N$ with the property $f(P)\subset \varphi(M)$, factors through $\varphi$.} called the \textit{leaves} of $\mathcal F(D)$, whose tangent spaces coincide with the fibres of $D$. 

Denote the $D$-valued vector fields on $\C$ by
\[\Gamma^\infty(D):=\set{X\in \vf(\C)\given X_g\in D_g\text{ for all }g\in\C}.\]
It is not hard to see that $D$ is \textit{locally of finite type}, i.e.\ for any $g\in \C$ there exist finitely many $X_i\in \Gamma^\infty(D)$ such that:
\begin{enumerate}
\item $D_g=\mathrm{Span}(X_i|_g)_i$,
\item For any $X\in \Gamma^\infty(D)$ there exists a neighborhood $U$ of $g$ in $\C$ and functions $f_i{}^j\in C^\infty(U)$, such that $[X,X_i]_h=\sum_j f_i{}^j(h)X_j|_h$ for all $h\in U$.
\end{enumerate}
Indeed, we may pick sections $(\alpha_i)_i$ of $A^L(\C)$, such that they constitute a local frame on some neighborhood $V$ of $s(g)$ in $\X$, and extend them to left-invariant vector fields $(X_i)_i$. The point (i) is clearly satisfied; to show (ii), denote $U=s^{-1}(V)$, and note that that the set of left-invariant vector fields is closed under the Lie bracket by Lemma \ref{lem:closed_lie_bracket}, so on $U$ there holds
\[
[X_i,X_j]=\textstyle\sum_k f_{ij}{}^k X_k,
\]
for some functions $f_{ij}{}^k\in C^\infty(U)$, and since any $X\in\Gamma^\infty(D|_U)$ can be written as a $C^\infty(U)$-linear combination of the tuple $(X_i|_U)_i$, the rest follows by using the Leibniz rule for the Lie bracket. 

Since $D$ is locally of finite type, it is integrable by \cite{sussman}. The following proposition says that the leaves are precisely the connected components of the orbits \[\set*{\underline s(\underline t^{-1}(g))\given g\in\C}\] of the groupoid $\C \tensor[_s]{\times}{_t} \G(\C)\rra\C$ corresponding to the right action of $\G(\C)$ on $s\colon\C\rightarrow \X$.

\begin{prop}
Let $\C\rra\X$ be a Lie category without boundary. The integral manifold of the singular distribution $D\subset T\C$ through $g\in \C$ is $L_g(\G(\C)^{s(g)})$.
\end{prop}
\begin{proof}
Corollary \ref{cor:ranks_invertibles} states $L_g|_{\G(\C)^{s(g)}}$ has constant rank, so rank theorem can be applied to deduce $L_g(\G(\C)^{s(g)})$ is an immersed submanifold of $\C^{t(g)}$, the tangent space at $gh$ of which is the image of differential $\d(L_g)_h$, for any $h\in \G(\C)^{s(g)}$.
\end{proof}

\begin{rem}
	In the case when $\C$ has a boundary, the last proposition is in general not true, as is easily seen by considering the order category on $\R$. In general, a natural candidate for the integral manifold of $D$ through $g$ is $L_g(U)$ for an appropriate small open neighborhood of $1_{s(g)}$ in $\C^{s(g)}$, but proving such a general result is harder since rank theorem is not true for manifolds with boundaries (without additional assumptions on $L_g$).
\end{rem}

\section{Extensions of Lie categories to Lie groupoids}
\label{sec:ext}
We have witnessed interesting examples of Lie categories appear by restricting to an embedded subcategory of a Lie groupoid, in spirit of Lemma \ref{lem:lie_subcat}. As we will see next, in order to ascertain the main properties of such Lie categories, the assumption of being embedded may be weakened to being immersed. 

\begin{defn}
An \textit{extension to a Lie groupoid} of a Lie category $\C\rightrightarrows \X$ is a Lie groupoid $\G\rightrightarrows \X$, together with an injective, immersive functor $F\colon \C\rightarrow \G$ over the identity. In other words, $\G$ is a Lie groupoid such that $\C$ is its wide Lie subcategory. If such an extension exists, we say that $\C$ is \textit{extendable to a Lie groupoid}. Furthermore, we say that an extension to a Lie groupoid is \textit{weakly étale}, if $\dim \G=\dim \C$.
\end{defn}
\begin{rem}
Note that if $\C$ and $\G$ do not have a boundary, then an extension is weakly étale if, and only if, the map $F$ is étale (i.e.\ a local diffeomorphism), by virtue of inverse map theorem. If $\dim \G=\dim \C$, but $F(\partial \C)\not\subset \partial\G$, then $F$ is not a local diffeomorphism (see e.g.\ the order category from Example \ref{ex:order_cat}).
\end{rem}

An obvious necessary condition for an arbitrary category $\C$ to admit an injective functor $F$ into a groupoid $\G$, is \textit{cancellativity} of all elements in $\C$, i.e.\ all left and right translations in $\C$ must be injective. For example, if $gh=gk$ holds for some morphisms in $\C$, then $F(g)F(h)=F(g)F(k)$ holds in $F(\C)\subset \G$, implying $F(h)=F(k)$, and hence $h=k$ by injectivity of $F$, so $L_g$ is injective.

In the differentiable setting, extensions to groupoids reflect on ranks and algebroids; the following lemma yields necessary conditions on a Lie category $\C$ to be extendable to a Lie groupoid.
\begin{lem}
\label{lem:extensions}
If a Lie category $\C$ is extendable to a Lie groupoid $\G$, all its morphisms have full and constant rank, and all left and right translations are injective. Moreover, if the extension is weakly étale, then $A^L(\C)\cong A(\G)\cong A^R(\C)$.
\end{lem}
\begin{proof}
Let $F\colon \C\rightarrow \G$ be the groupoid extension, and let $g,h\in \C$ be composable. To prove that $g$ has full and constant rank, denote the left translation in $\C$ by $g$ as $L_g^\C$, and the left translation in $\G$ by $F(g)$ as $L^\G_{F(g)}$. By functoriality, the diagram
\[\begin{tikzcd}[column sep=2.5em]
	{\C^{s(g)}} & {\C^{t(g)}} \\
	{\G^{s(g)}} & {\G^{t(g)}}
	\arrow["{L_g^\C}", from=1-1, to=1-2]
	\arrow["{L_{F(g)}^\G}"', from=2-1, to=2-2]
	\arrow["F"', from=1-1, to=2-1]
	\arrow["F", from=1-2, to=2-2]
\end{tikzcd}\]
commutes, so if $v\in T_h(\C^{s(g)})$ is such that $\d(L_g^\C)_h(v)=0$, it implies $\d F_h(v)=0$, and so $v=0$ since $F$ is an immersion. This proves that $\d(L_g^\C)_h$ has full rank, and a similar proof shows an analogous result for the right translation. Hence $g$ has full and constant rank.

To prove the second part, denote the source and target maps in $\C$ and $\G$ as $s^\C, t^\C$ and $s^\G,t^\G$, respectively. Commutativity of the following diagram,
\[\begin{tikzcd}[column sep=1.25em]
	\C && \G \\
	& \X
	\arrow["F", from=1-1, to=1-3]
	\arrow["{t^\C}"', from=1-1, to=2-2]
	\arrow["{t^\G}", from=1-3, to=2-2]
\end{tikzcd}\]
together with the assumption that our extension is weakly étale, implies that for any $x\in \X$, $\d F_{1_x}$ maps $\ker \d t^\C_{1_x}$ isomorphically onto $\ker \d t^\G_{1_x}$, so $F$ induces an isomorphism $F_*^L\colon A^L(\C)\rightarrow A^L(\G)$ of vector bundles, and similarly between $A^R(\C)$ and $A^R(\G)$. Since $F$ is a morphism of Lie categories, these are in fact isomorphisms of Lie algebroids, so we yield the wanted chain of isomorphisms.
\end{proof}

\begin{cor}
If a Lie category is extendable to a Lie groupoid, then the composition map $m\colon\comp\C\rightarrow \C$ is a submersion.
\end{cor}
\begin{proof}
Follows directly from Corollary \ref{cor:composition_submersion} and Lemma \ref{lem:extensions}.
\end{proof}

\begin{rem}
To conclude $A^L(\C)\cong A^R(\C)$, it is enough to replace the assumption on the extension $F\colon \C\rightarrow \G$ being weakly étale with the following weaker condition: 
\begin{align}
\label{eq:nice_extension}
\d(\inv)_{1_x}(\d F_{1_x}(\ker \d t_{1_x}^\C))\subset\d F_{1_x}(\ker \d s^\C_{1_x}),
\end{align}
but notice that if $\dim \G\neq \dim \C$, then the left and right algebroids of $\C$ will not be isomorphic to the algebroid $A(\G)$, since their ranks will differ. For simpler notation, assume that $\C\subset \G$ and $F\colon \C\hookrightarrow\G$ is an injective immersion.

To show that $\d(\inv)_{1_x}(\ker \d t_{1_x}^\C)\subset\ker \d s^\C_{1_x}$ implies $A^L(\C)\cong A^R(\C)$, first observe that the assumption \eqref{eq:nice_extension} implies equality instead of just the inclusion. Hence the inversion map induces a vector bundle isomorphism \[\inv_*\colon A^L(\C)\rightarrow A^R(\C),\]
and now it is not difficult to see that for any $\alpha\in \Gamma^\infty(A^L(\C))$, its left-invariant extension $\alpha^L$ to whole $\G$ is $\inv$-related to the right-invariant extension $(\inv_*\alpha)^R$. Hence if $\beta\in\Gamma^\infty(A^L(\C))$ is another section, we obtain that for any $g\in\G$,
\[
\d(\inv)_g([\alpha^L,\beta^L]_g)=[(\inv_*\alpha)^R,(\inv_*\beta)^R]_{g^{-1}}.
\]
Taking $g=1_x$ shows that $\inv_*$ preserves the brackets of $A^L(\C)$ and $A^R(\C)$, so it is an isomorphism of Lie algebroids.
\end{rem}

\begin{rem}
\label{rem:questions}
The importance of Lemma \ref{lem:extensions} is in the fact that it enables us to easily provide positive answers to the following questions, for the case of Lie categories extendable to groupoids:
\begin{enumerate}
\item Are Lie monoids parallelizable?
\item For fixed objects $x,y\in\X$ of a Lie category $\C\rra\X$, is the set $\C_x^y$ of morphisms from $x$ to $y$ a submanifold of $\C$? Equivalently, do Hom-functors map into the category $\mathbf{Diff}$ instead of just $\mathbf{Set}$?
\end{enumerate}
We remark that these questions remain open for Lie categories which do not admit an extension to a Lie groupoid.
\end{rem}

\begin{cor}
If a Lie monoid $M$ is extendable to a Lie group, it is paralellizable.
\end{cor}
\begin{proof}
Extending a basis of $T_eM$ to a tuple of left-invariant vector fields yields a global frame for $TM$ since all elements of $M$ have full rank by Lemma \ref{lem:extensions}.
\end{proof}

The proof of the positive answer to (ii) is similar as in the Lie groupoid case:
\begin{cor}
\label{cor:hom_sets_embedded}
If a Lie category $\C\rra\X$ is extendable to a Lie groupoid, then for any $x,y\in\X$ the set $\C_x^y$ is a closed embedded submanifold of $\C$. In particular, $\C_x^x$ is a Lie monoid for any $x\in \X$.
\end{cor}
\begin{proof}
We realize $\C_x^y$ as an integral manifold of a certain distribution on $\C_x$, namely $D=\ker \d(t|_{\C_x})$, or more instructively, $D_g=\ker \d t_g\cap \ker \d s_g$ for all $g\in \C_x$. This is a regular distribution on $\C_x$, since $D_g=\d(L_g)_{1_x}(D_{1_x})$ holds -- the latter is a consequence of the equality $s|_{\C^{t(g)}}\circ L_g=s|_{\C^x}$ and $g$ having full left rank by Lemma \ref{lem:extensions}, which moreover implies that $D$ is a trivial vector subbundle of $T \C_x$. 

Since $D$ is the kernel of a differential of a smooth map, it is involutive, so by Frobenius' theorem integrable. The leaves of the corresponding foliation are the connected components of subspaces $\set{\C_x^y\given y\in \X}$ of $\C_x$, so they are its initial submanifolds. Since the subspaces $\C_x^y$ are also closed in $\C_x$, they are embedded.
\end{proof}

\section{Completeness of invariant vector fields}
It is well-known that on any Lie group $G$, left-invariant vector fields are complete. In this section, we generalize this result to Lie monoids with normal boundaries. Furthermore, we generalize the characterization of completeness of left-invariant vector fields on Lie groupoids to Lie categories with normal boundaries. At last, we will discuss the exponential map for Lie monoids.

\begin{defn}
Let $X$ be a vector field on a smooth manifold $M$ with or without boundary, and denote by $J^X_g$ the maximal interval on which the integral path $\gamma^X_g$ of $X$, starting at $g\in M$, is defined. We say that $X$ is \textit{half-complete} if for any $g\in M$ there either holds $[0,\infty)\subset J^X_g$ or $(-\infty,0]\subset J^X_g$.
\end{defn}

\begin{rem}
\label{rem:inward}
In what follows, we will assume the reader is familiar with the notion of inward-pointing and outward-pointing tangent vectors on the boundary; we direct to \cite[p.\ 118]{lee} for a basic reference. A particularly useful observation is that a vector field $X\in\vf(M)$ which is inward-pointing (or outward-pointing) at a certain point $g\in \partial M$, must remain inward-pointing (or outward pointing) on a neighborhood of $g$ in $\partial M$, from which it follows that $J_g^X$ cannot contain an open neighborhood of zero, but may only contain a half-closed interval $[0,\varepsilon)$, for some $\varepsilon>0$ (or $(-\varepsilon,0]$ in the outward-pointing case).
\end{rem}

Recall from Definition \ref{def:normal_bdry} that a Lie monoid $M$ has a normal boundary, if either $e\in\Int M$, or we have both that $e\in\partial M$ and $\partial M$ is a submonoid of $M$.

\begin{thm}
\label{thm:complete}
Let $M$ be a Lie monoid with a normal boundary and let $X$ be a left-invariant vector field on $M$. The following holds:
\begin{enumerate}
\item Suppose either that $e\in\Int M$, or that $e\in\partial M$ and $X_e$ is tangent to $\partial M$. Then $X|_{\partial M}$ is tangent to $\partial M$, and $X$ is complete.
\item Suppose $e\in\partial M$ and $X_e$ is either inward-pointing or outward-pointing. Then either $J_e^X=[0,\infty)$ or $J_e^X=(-\infty,0]$, respectively, and $X$ is half-complete.
\end{enumerate}
Moreover, the flow of $X$ is given for all $t\in J_e^X$ by $\phi_t^X=R_{\phi_t^X(e)}$.
\end{thm}
\begin{proof}
We first inspect the assumptions from (i): notice that if $e\in\Int M$, then $J_e^X$ clearly contains an open interval around zero; on the other hand, if $e\in\partial M$ and $X_e$ is tangent to $\partial M$, then $X|_{\partial M}$ must be everywhere tangent to $\partial M$ by left-invariance of $X$ and the fact that $\partial M$ is a Lie submonoid of $M$, so $J_e^X$ must again contain an open interval around zero. Furthermore, the assumption from (ii) that $X_e$ is either inward-pointing or outward-pointing implies that $J_e^X$ contains a half-open interval of the form $[0,\varepsilon)$ or $(-\varepsilon,0]$, respectively.

We next observe that for any $g\in M$, the composition $L_g\circ \gamma_e^X$ is an integral path of $X$ starting at $g$, so maximality of $J_{g}^X$ implies
\begin{align}
\label{eq:eunderg}
J_e^X\subset J_g^X,
\end{align} 
and also $\phi_t^X(g)=g\gamma_e^X(t)=R_{\gamma_e^X(t)}(g)$ for all $t\in J_e^X$. Notice that \eqref{eq:eunderg} now implies that in the case $e\in\Int M$, $J_g^X$ contains an open interval around zero for all $g\in M$, and in particular this holds for any $g\in \partial M$, so we conclude that $X|_{\partial M}$ must be tangent to $\partial M$, by virtue of Remark \ref{rem:inward}.

Let $\tau\in J_e^X$ and consider the affinely translated path 
\[\zeta^\tau\colon (J_e^X-\tau)\rightarrow M, \quad \zeta^\tau (t)=\gamma_e^X(t+\tau).\]
Since the maximal domain of an affinely translated integral path is just the affinely translated maximal domain, we get $J_e^X-\tau=J_{\zeta^\tau(0)}^X$.\footnote{We do not need the Lie monoid structure to prove this simple fact.} 
Together with \eqref{eq:eunderg}, this implies
\[
J_e^X+\tau\subset J_e^X,\ \text{for all }\tau\in J_e^X.
\]
This implies that if $J_e^X$ contains an open interval around zero, it must equal $\R$, and if $[0,\varepsilon)\subset J_e^X$ for some $\varepsilon>0$, then $[0,\infty)\subset J_e^X$; similarly for the outward-pointing case. Together with equation \eqref{eq:eunderg}, this proves our claims regarding completeness and half-completeness.
\end{proof}

\begin{cor}
If $M$ is a Lie monoid and $e\in \Int M$, then elements in $\partial M$ do not have full rank. Hence if also $\partial M\neq \emptyset$, then $M$ is not extendable to a Lie group.
\end{cor}
\begin{proof}
For a proof by contradiction, assume that there is a $g\in\partial M$ with full rank, and pick any inward-pointing vector $v\in T_g M$. Since $g$ has full rank, the vector $v$ is extendable to a unique left-invariant vector field $X$ on $M$, but now $X|_{\partial M}$ must be tangent to $\partial M$ by Theorem \ref{thm:complete}, contradicting our assumption that $v$ is inward-pointing. The second part of the corollary follows from Lemma \ref{lem:extensions}.
\end{proof}

An easy example of the last corollary in play is the Lie monoid $[0,\infty)$ for multiplication. This result may easily be generalized; in any Lie category $\C\rra\X$, a similar inclusion as \eqref{eq:eunderg} holds:
\begin{align}
\label{eq:unitunderg}
J_{1_{s(g)}}^X\subset J_g^X,
\end{align}
for any $g\in \C$ and any left-invariant vector field $X\in\vf^L(\C)$. The same argument as before shows that if $u(\X)\subset \Int\C$, then $X|_{\partial\C}$ is tangent to $\partial \C$, so we similarly obtain: 

\begin{cor}
If $\C\rra\X$ is a Lie category with $u(\X)\subset \Int \C$, then morphisms in $\partial \C$ do not have full rank. Hence if $u(\X)\subset\Int\C$ and $\partial \C\neq \emptyset$, then $\C$ is not extendable to a Lie groupoid.
\end{cor}

\begin{proof}
As before, suppose there is a morphism $g\in\partial\C$ with full rank. Regularity of the boundary implies $\partial (\C^{t(g)})=\partial\C\cap \C^{t(g)}$, so we may pick an inward-pointing vector $v\in \ker \d t_g$. Since $g$ has full rank, we can form the vector $\d(L_g)_{1_{s(g)}}^{-1}(v)\in \ker\d t_{1_{s(g)}}$, extend it to a section of $u^*\ker\d t$ using partitions of unity, and finally extend it to a left-invariant vector field $X$ on $\C$. Since $X_g=v$, we arrive to a contradiction to the fact that $X|_{\partial \C}$ is tangent to $\partial \C$. 
\end{proof}



\begin{rem}
\label{rem:gpd_empty_bdry}
In the case when $g$ is an invertible morphism of a Lie category, the inclusion \eqref{eq:unitunderg} is actually an equality, which follows from the fact that $L_{g^{-1}}\circ\gamma^X_g$ is an integral path of $X$ starting at $1_{s(g)}$. 
\end{rem}

We now present the promised characterization of completeness of invariant vector fields on a Lie category. The proof is a small adaptation of the one from the theory of Lie groupoids.

\begin{prop}
Let $\C\rra\X$ be a Lie category with a normal boundary, and let $\alpha\in \Gamma^\infty(A^L(\C))$ be a section of its left Lie algebroid. Suppose either $u(\X)\subset \Int \C$, or that $u(\X)\subset\partial \C$ and $\alpha^L|_{u(\X)}$ is tangent to $\partial \C$. Then $\alpha^L\in\vf^L(\C)$ is complete if, and only if, $\rho^L(\alpha)\in\vf(\X)$ is complete.
\end{prop}
\begin{proof}
First note we have already argued that under given assumptions, the restriction $\alpha^L|_{\partial \C}$ is tangent to $\partial\C$, implying that the maximal domain of any integral path of $\alpha^L$ is an open interval.

For the forward implication, note that $\alpha^L$ and $\rho^L(\alpha)$ are $s$-related, so that if $\phi^{\alpha^L}_t$ is defined for some $t\in\R$, then so is $\phi^{\rho^L(\alpha)}_t$, by the fact that $s$ is a surjective submersion.
\[\begin{tikzcd}
	\C & \C \\
	\X & \X
	\arrow["s"', from=1-1, to=2-1]
	\arrow["{\phi^{\alpha^L}_t}", from=1-1, to=1-2]
	\arrow["{\phi^{\rho^L(\alpha)}_t}"', from=2-1, to=2-2]
	\arrow["s", from=1-2, to=2-2]
\end{tikzcd}\]
Indeed, take an open cover $(U_i)_i$ of $\X$ by domains of local sections $\sigma_i\colon U_i\rightarrow \C$ of $s$, and define $\vartheta_i\colon \R\times U_i\rightarrow \X$ as $\vartheta_i(t,x)=(s\circ\phi_t^{\alpha^L}\circ\sigma)(x)$. It is straightforward to check that for any $x\in U_i$, the map $t\mapsto \vartheta_i(t,x)$ is the integral path of $\rho^L(\alpha)$ starting at $x$, so the maps $\vartheta_i$ collate to the global flow $\phi^{\rho^L(\alpha)}\colon\R\times\X\rightarrow\X$ of $\rho^L(\alpha)$, by uniqueness of integral paths.

For the converse implication, suppose $\rho^L(\alpha)$ is complete, and let $\gamma\colon (a,b)\rightarrow \C$ be an integral path of $\alpha^L$; then $s\circ\gamma$ is an integral path of $\rho^L(\alpha)$, thus $s\circ\gamma$ admits a unique extension to an integral path of $\rho^L(\alpha)$ defined on whole $\R$, which we again denote by $s\circ\gamma$, so the expression $(s\circ\gamma)(b)$ is defined. There now exists a path $\delta\colon (b-\varepsilon,b+\varepsilon)\rightarrow \C$, which is an integral path of $\alpha^L$, such that $\delta(b)=1_{(s\circ\gamma)(b)}$; importantly, we may assume that $\varepsilon$ is small enough that $\delta$ maps into $\G(\C)$, since the latter is either open in $\Int\C$ or in $\partial \C$ by Theorem \ref{thm:open_inv}. We define our wanted extension $\bar\gamma\colon (a,b+\varepsilon)\rightarrow 
\C$ of $\gamma$ as
\[
\bar\gamma(t)=
\begin{cases}
\gamma(t)&\text{if }t\in(a,b),\\
\gamma(b-\frac{\varepsilon}2)\delta(b-\frac\varepsilon 2)^{-1}\delta(t)&\text{if }t\in(b-\varepsilon,b+\varepsilon).
\end{cases}
\]
Since both $s\circ\gamma$ and $s\circ\delta$ are integral paths of $\rho^L(\alpha)$ valued $(s\circ\gamma)(b)$ at $b$, they coincide on their common domain, so that $s(\gamma(b-\frac\varepsilon 2))=s(\delta(b-\frac\varepsilon 2))$, and moreover $\delta$ must lie in $\C^{(s\circ\gamma)(b)}$ since $\alpha^L$ is tangent to $t$-fibres, which altogether implies that the multiplication in the definition of $\bar \gamma$ is well-defined. Finally, $\bar \gamma$ is in fact an integral path of $\alpha^L$ since for any $t\in (b-\varepsilon,b+\varepsilon)$ there holds:
\[
\bar\gamma'(t)=\d{\big(L_{\gamma(b-\frac\varepsilon 2)\delta(b-\frac\varepsilon 2)^{-1}}\big)}_{\delta(t)}(\delta'(t))=\alpha^L_{\gamma(b-\frac\varepsilon 2)\delta(b-\frac\varepsilon 2)^{-1}\delta(t)}=\alpha^L_{\bar\gamma(t)}.
\]
By uniqueness of integral paths, the two partial definitions of $\bar\gamma$ coincide on their common domain, from which we conclude that $\bar\gamma$ indeed extends $\gamma$. Similarly can be done for the other endpoint, showing that the integral path $\gamma$ can be extended to the whole $\R$.
\end{proof}

Let us now turn back to Lie monoids. As a corollary of Theorem \ref{thm:complete}, we can generalize the exponential map from the theory of Lie groups to Lie monoids. First off, the following says that when $e\in\Int M$, the map $T_eM\rightarrow M$, defined as $v\mapsto \phi^{v^L}_1(e)$ is just the usual exponential map on $G(M)$.

\begin{cor}
\label{cor:exp_grp}
Let $M$ be a Lie monoid with $e\in \Int M$. The image of the integral path $\gamma_e^X$ of any left-invariant vector field $X$ on $M$ is contained in the core $G(M)$.
\end{cor}
\begin{proof}
By Theorem \ref{thm:complete}, there holds $\phi_{-t}^X(e)\phi_t^X(e)=\phi_t^X(\phi_{-t}^X(e))=e$ for any $t\in \R$, so $\phi_{-t}^X(e)$ is the inverse of $\phi_t^X(e)$.
\end{proof}

On the other hand, Theorem \ref{thm:complete} also enables us to define the exponential map for Lie monoids with boundaries:
\begin{defn}
Let $M$ be a Lie monoid such that $\partial M$ is its submonoid, $e\in\partial M$. Let $T_e^{\plus}M\subset T_eM$ denote the subset consisting of inward-pointing vectors in $T_e M$ and the vectors in $T_e(\partial M)$. The \textit{exponential map} on $M$ is then defined as
\begin{align*}
\exp\colon T_e^{\plus} M\rightarrow M,\quad \exp(v)=\phi^{v^L}_1(e).
\end{align*}
\end{defn}
\begin{rem}
In any boundary chart centered at $e$, $T_e^{\plus}M$ is identified with the closed upper half-space $\mathbb H^{\dim M}$. We observe that $T_e^{\plus}M$ possesses an algebraic structure of a semimodule over a semiring $[0,\infty)$.\footnote{A semiring $R$ satisfies all the axioms of a ring, except the existence of additive inverses. Because of this, we must additionally impose $0\cdot a=0=a\cdot 0$ for all $a\in R$. We amend modules to obtain semimodules in precisely the same way. } 

As before, Corollary \ref{cor:exp_grp} ensures that the restriction of $\exp$ to the vectors tangent to $\partial M$, is precisely the usual exponential map of the Lie group $G(M)$.
\end{rem}

Similar results to those from the theory of Lie groups can be obtained for the exponential map as defined above, by using an identical approach to the respective proofs, but working with one-sided derivatives. We leave the following results for the reader as an exercise:
\begin{enumerate}
\item Let $v\in T^{\plus}_eM$ and $t\geq 0$. By rescaling lemma, $\phi^{v^L}_{tr}(e)=\phi_r^{tv^L}(e)$ for all $r\geq 0$, and setting $r=1$ we obtain
\[
\exp(tv)=\phi_t^{v^L}(e).
\]
\item The map $\exp$ is a smooth, and there holds $\d{(\exp)}_e=\id_{T_eM}$. Since $\exp$ maps $\exp(T_e(\partial M))\subset \partial M$ by Theorem \ref{thm:complete} (i), one can apply the inverse map theorem for maps between manifolds with boundaries to conclude that $\exp$ is a local diffeomorphism at the point $0\in T_e^{\plus} M$. 
\item Considering the Lie monoid $[0,\infty)$ for addition, any smooth homomorphism $\alpha\colon [0,\infty)\rightarrow M$ of Lie monoids is called a \textit{one-parametric submonoid} of $M$. For any $v\in T_e^{\plus}M$, the map $t\mapsto \exp(tv)$ is a one-parametric submonoid of $M$, and any one-parametric submonoid $\alpha$ can be written as $\alpha(t)=\exp(t\dot\alpha(0))$, where $\dot\alpha(0)$ denotes the one-sided derivative of $\alpha$ at zero. 
\end{enumerate} 

A consequence of point (iii) above is \textit{naturality} of $\exp$, i.e.\ if $\phi\colon M\rightarrow N$ is a morphism between Lie monoids with normal boundaries, such that the units of $M$ and $N$ are contained in the respective boundaries, then for any $v\in T_e^{\plus}M$, the map $\alpha(t)= \phi(\exp_M(tv))$ defines a one-parametric submonoid of $N$ with $\dot\alpha(0)=\d\phi(v)$, so we obtain that the following diagram commutes.
\[\begin{tikzcd}
	M & N \\
	{T_e^{\plus}M} & {T_e^{\plus}N}
	\arrow["\phi", from=1-1, to=1-2]
	\arrow["\d\phi", from=2-1, to=2-2]
	\arrow["{\exp_M}", from=2-1, to=1-1]
	\arrow["{\exp_N}"', from=2-2, to=1-2]
\end{tikzcd}\]

\section{An application to physics: Statistical Thermodynamics}
\label{sec:std}
The interpretation that morphisms correspond to physical processes, and objects to physical states, can be applied to yield a rigorous approach to statistical physics, which we will now demonstrate. We will first focus purely on categorical aspects, and then consider differentiability.

Suppose we are given an isolated physical system consisting of an unknown number of particles, each of which can be in one of the $n+1$ a-priori given \textit{microstates}, which we will index by
\[
i\in\set{0,\dots,n}.
\]
Since the number of particles in our system is unknown and often large, we need to work with tuples of probabilities $(p_0,\dots,p_n)$, where each $p_i$ is the probability that a particle, chosen at random, is in the $i$-th microstate. Any tuple of probabilities $(p_i)_{i=0}^n$ is subjected to the constraint
\[
\textstyle\sum_i p_i=1,
\]
and we will refer to any such tuple $p=(p_i)_i$ as a \textit{configuration} of the system, which is just a probability distribution on the finite set of microstates above. The set of all configurations of our system is thus the standard $n$-simplex,
\[
\Delta^{n}=\set[\big]{(p_0,\dots,p_n)\in [0,1]^{n+1}\given \textstyle\sum_i p_i=1}.
\]
We associate to any configuration $(p_i)_i$ of our system its expected surprise,
\[
S(p_i)_i=-\sum_i p_i\log p_i,
\]
which is called the \textit{entropy} of the configuration $(p_i)_i$. Letting $f(x)=x\log x$, we find $\lim_{x\rightarrow 0^{\plus}}f(x)=0$, so $f$ may be extended to $[0,\infty)$ by defining $f(0)=0$, implying that $S\colon \Delta^{n}\rightarrow \R$ is defined on whole $\Delta^n$ and continuous. 

The construction of the space of morphisms between different configurations of our system is now an application of the \textit{second law of thermodynamics}: 
\begin{quote}
\vspace{0.5em}
\textit{A process in an isolated physical system is feasible if, and only if, the change of entropy pertaining to the process is non-negative.}
\vspace{0.5em}
\end{quote}
In accord with the second law, we define
\[
\D=\set*{(p_i)_i\rightarrow (q_i)_i\given S(q_i)_i- S(p_i)_i\geq 0}.
\]
In other words, $\D$ consists of pairs $(q,p)\in \Delta^{n}\times \Delta^{n}$ of configurations, such that the entropy of the target configuration $q$ is no less than that of the source $p$. That $\D\rra \Delta^{n}$ is a category follows from the fact that the map
\[
\delta S\colon \Delta^{n}\times\Delta^{n}\rightarrow \R,\quad \delta S(q,p)=S(q)-S(p),
\]
is a functor\footnote{This is aligned with the moral that entropy should be inherently perceived as a categorical concept, which was first adopted by Baez et al.; $\delta S$ is in fact the only map (up to a multiplicative scalar) which is functorial, convex-linear and continuous, see \cite{baez2011} for details.} from the pair groupoid of $\Delta^{n}$ to the group $\R$ for addition. Notice that the invertible morphisms in $\D$ are precisely $\delta S^{-1}(0)$, which are just the processes with zero entropy change.

In the differentiable setting, we need to make certain adjustments to our category $\D\rra \Delta^{n}$, since it is not a Lie category. First, we note that $\Delta^n$ is a manifold with corners, and $S$ is not smooth at its boundary $\partial\Delta^{n}$ since $\lim_{x\rightarrow 0^{\plus}}(x\log x)'=-\infty$, which is why we first restrict our attention to the interior $\Int\Delta^{n}.$ 
Secondly, we notice that the category $\D\rra\Delta^{n}$ has a terminal object, as shown by the following.

\begin{claim}
The only critical point of entropy $S|_{\Int\Delta^{n}}$ is given by the so-called microcanonical configuration, i.e.\ $p_i^{\mu}=\frac 1{n+1}$ for all $i$. This is a maximum of $S$.
\end{claim}
\begin{proof}
Constraint $\sum_i p_i=1$ implies $\sum_i \d p_i=0$ and $\d p_0=-\sum_{i=1}^{n}\d p_i$, so we have
\[
\d S_{(p_i)_i}=-\sum_i(1+\log p_i)\d p_i=-\sum_i\log p_i \d p_i=-\sum_{i=1}^{n}(\log p_i-\log p_0)\d p_i,
\]
which vanishes if, and only if, $p_i=p_0$ for all $i$, i.e.\ $p_i=\frac 1{n+1}$. That this is a local maximum is left as an exercise, and it is not hard to see that the value of $S$ at $(p_i^\mu)_i$ is greater than the value of $S$ on $S|_{\partial\Delta^{n}}$.
\end{proof}

The configuration $p^\mu$ has little physical importance -- for example, we will see below that it can be interpreted as the configuration that is attained at thermodynamical equilibrium at infinite temperature. We will thus remove it from the interior of our $n$-simplex of objects, and define our space of objects to be
\[
\X=\Int\Delta^n- \set{p^\mu},
\]
which is a smooth manifold without boundary. Moreover, we define the space of morphisms over $\X$ as the set 
\[
\C=(\delta S|_{\X\times\X})^{-1}([0,\infty))=\set*{(q,p)\in\X\times\X\given S(q)- S(p)\geq 0}.
\] 
As before, $\C\rra \X$ is a subcategory of the pair groupoid on $\X$. By virtue of Example \ref{ex:preimage_subcat}, to prove that $\C\rra\X$ is a Lie category, we first need to check $\delta S|_{\X\times\X}$ has a regular value 0. By above claim, the only critical point of the map $\delta S$ is the identity morphism $(p^\mu,p^\mu)\in\D$, which is not in $\C$.

Secondly, we have to show that $\C\rra\X$ has a regular boundary, i.e.\ that $s|_{\partial\C},t|_{\partial \C}$ are submersions. To this end, first note that for any $(q,p)\in\partial \C$, 
\begin{align*}
&T_{(q,p)}\partial \C=\ker\d(\delta S)_{(q,p)}\\
&=\set*{(v,w)\in T_q\X\oplus T_p \X\given \textstyle\sum_{i=1}^{n}(\log q_i-\log q_0)v_i=\sum_{i=1}^{n}(\log p_i-\log p_0)w_i}.
\end{align*}
Let $v\in T_q \X$. Since $p_i\neq p_0$ for some $i$, we define
\[
w_i=\frac{\textstyle\sum_{i=1}^{n}(\log q_i-\log q_0)v_i}{\log p_i-\log p_0}
\]
and now we let $w=(w_0,0,\dots,0,w_i,0,\dots,0)$ where $w_0=-w_i$ is set to ensure $w\in T_p\X$, thus we obtain the wanted pair $(v,w)\in T_{(q,p)}\partial\C$ with $\d t(v,w)=v$; we similarly show that $s|_{\partial \C}$ is a submersion. Furthermore, $\C\rra\X$ satisfies all the conclusions from Lemma \ref{lem:extensions}, since the pair groupoid $\X\times \X$ is its weakly étale extension; in particular, the left and right Lie algebroids of $\C$ are isomorphic to $T\X\approx\X\times \R^n$.

The question interesting for physics is: what is the configuration $p^{\epsilon}\in\X$ at which the system attains a thermodynamical equilibrium? It is well-known that the answer to this question is obtained using the so-called \textit{Gibbs algorithm}, which we provide here for completeness. To this end, we need additional a-priori given data, namely each microstate $i$ of our system has an a-priori assigned quantity $E_i$, called the \textit{energy} of $i$-th microstate. To derive the wanted configuration $p^\epsilon$ we utilize the so-called \textit{principle of maximum entropy}:

\begin{quote}
\vspace{0.5em}
\textit{A state is at thermodynamical equilibrium if, and only if, it maximizes the entropy with respect to the systemic constraints.}
\vspace{0.5em}
\end{quote}
That is, we must find the constrained extremum of $S$ with respect to constraints $\sum_i p_i=1$ and $\sum_i p_i E_i=E(p_i)_i$. Here $E\colon\X\rightarrow \R$ is a function on the object space, determined by the first law of thermodynamics up to an additive constant, as we will see below -- this will impose thermodynamical considerations onto the statistical description of our system.

To apply the method of Lagrange multipliers, define the function $\hat S\colon \X\rightarrow \R$,
\[
\hat S(p_i)_i=S(p_i)_i-\lambda_1\left(E(p)-\textstyle\sum_i p_i E_i\right)-\lambda_2(1-\textstyle\sum_i p_i).
\]
Requiring $\frac{\partial \hat S}{\partial p_i}=0$ yields 
$
p_i=\frac 1 Ze^{\lambda_1 E_i},
$
where we have defined $Z=e^{-(1+\lambda_2)}$. The constraint 
$
\textstyle\sum_i p_i=1$ then reads 
\[
Z=e^{-(1+\lambda_2)}=\textstyle\sum_i e^{\lambda_1 E_i}.
\]
On the other hand, the first law of thermodynamics for systems where no work is exerted reads
\[
\d E=kT\d S
\]
where $T$ is the temperature (a macroscopic external constant) at which the system is held, and $k$ denotes the Boltzmann constant. Writing out the total differentials $\d E$ and $\d S$ gives $\lambda_1=-\frac 1 {kT}$, and so we finally obtain the \textit{equilibrium configuration}:
\[
p_i^\epsilon=\frac 1 {Z} e^{-\frac{E_i}{kT}},\quad Z=\sum_i e^{-\frac{E_i}{kT}}.
\]
We observe that in this categorical framework, the configurations from which it is possible to attain the equilibrium configuration $p^\epsilon$, are now simply expressible as 
$
s(t^{-1}({p^\epsilon}))=\set*{p\in\X\given S(p^\epsilon)\geq S(p)}.
$

\section*{Further research}
Although we hope to have succeeded in portraying the richness of Lie categories and their potential in physics, we admit that we have not exhausted all research options regarding them. We state some of them here, and note that they provide possibilities for future research.

\begin{enumerate}
\item As stated in Remark \ref{rem:questions}, the question remains whether all Lie monoids are parallelizable, and whether $\C_x^y$ is a smooth manifold for given objects $x,y\in\X$ of a Lie category $\C\rra\X$.
\item Does there exist a class of Lie algebroids that cannot be integrated to a Lie groupoid, but can be integrated to a Lie category? 
\item Remark \ref{rem:base_boundary} shows the need for considering Lie categories whose object manifold has a boundary, or more generally, corners. Another example of such a Lie category should be the flow of a smooth vector field on a manifold with boundary, generalizing the flow Lie groupoid of a vector field on a boundaryless manifold. 
\item Infinite-dimensional Lie categories. The exterior bundle $\Lambda(E)$ of a vector bundle $E$ is an example of a bundle of Lie monoids, and it is moreover a subcategory of the tensor bundle $\oplus_{k=0}^\infty\otimes^k E$, which is a bundle of monoids that fails to have finite-dimensional fibres. Moreover, in statistical mechanics, physicists often work with an infinite number of microstates, and in quantum mechanics with infinite-dimensional Hilbert spaces. These examples show the need for introducing infinite-dimensional Lie categories. Since the theory of infinite-dimensional (Banach) manifolds with corners is already well-developed in \cite{corners}, a theory of infinite-dimensional Lie categories with corners seems realizable.
\item Multiplicative differential forms on Lie categories. On Lie groupoids, such structures can be used to describe integrated counterparts of Poisson structures \cite{poisson}, and more recently they have been used to provide a natural generalization of connections on principle bundles in \cite{mec}. We suspect that interesting geometric structures can be described with multiplicative differential forms on Lie categories.
\item Haar systems on Lie categories. Just as Haar systems on Lie groupoids provide a generalization of a Haar measure on a Lie group (and connects the theory of Lie groupoids to noncommutative geometry), a further generalization to Lie categories should provide a means of equipping the space $C_c(\C)$ of compactly supported functions on the space of arrows of a given category with the convolution product. In this fashion, one expects to generalize the construction of a groupoid C*-algebra, but due to the lack of existence of inverses of arrows, there is a-priori no natural way of obtaining the involution, hence the construction potentially generalizes only to a \textit{category Banach algebra}. We strongly suspect that a sensible notion of a Haar system on a Lie category will rather be defined in terms of a measure $\mu$ on the space $\comp\C$ of composible pairs of arrows, satisfying certain invariance and continuity conditions, instead of being defined in terms of left-invariant $t$-fibre supported measures on $\C$. Roughly speaking, the convolution on $C_c(\C)$ would then conceivably be defined by 
$
(f_1*f_2)(g)=\int_{m^{-1}(g)}f_1(g_1)f_2(g_2) \d \mu(g_1,g_2),
$
where the integration is done over the set $m^{-1}(g)=\set{(g_1,g_2)\in \comp \C\given g_1g_2=g}$ of composible arrows which compose to $g\in\C$.
	
An interesting question for mathematical physics might be whether the principle of least action in physics can be used within the theory of Lie categories, possibly using Haar systems, to describe existing or yield  novel physical theories.
\item Smooth sieves on Lie categories. In the context of Lie groupoids, the notion of a sieve is vacuous since all morphisms are invertible -- any sieve on an object must equal the whole $t$-fibre over that object. However, a notion of a smooth sieve seems possible for Lie categories, and their properties should be researched, together with the properties of Grothendieck sites of such sieves.
\item Dynamical systems. The theory of Lie groupoids has successfully been applied to describe dynamical systems, see e.g.\ \cite{dynamics}; a natural question is to what extent the theory of Lie categories presented in our paper can be used to generalize these results, to describe non-invertible dynamics.
\item Generalization of Morita equivalence from the context of Lie groupoids to Lie categories.
\end{enumerate}

\section*{Acknowledgments}
I express my gratitude to my supervisor Pedro Resende for his steady direction and patience; I also thank my second supervisor Ioan Mărcuț for his readiness to discuss Lie categories, for ideas on (counter-)examples, and for keeping a sharp eye on my proofs. Finally, I  thank Jure Kališnik and Rui L. Fernandes for additional fruitful ideas and talks.

\appendix
\section{Transversality on manifolds with corners}
\label{sec:transversality_corners}

We state here some basic definitions and results regarding finite-dimensional manifolds with corners that we need in the paper, mostly drawing from  \cite{corners}; in what follows, $X$ is assumed to be a second-countable topological space. Let $V$ be a real $n$-dimensional Banach space\footnote{Keep in mind that picking a basis of $V$ gives a homeomorphism $V\rightarrow\R^n$.}; denote by $\Lambda=(\lambda_i)_i$ a (possibly empty) set of linearly independent covectors $\lambda_i\in V^*$ and let \[V_\Lambda=\set{v\in V\given \lambda(v)\geq 0\text{ for all }\lambda\in\Lambda}.\]
Vacuously, there holds $V_\Lambda=V$ when $\Lambda = \emptyset$. We will call $\Lambda$ a \textit{corner-defining system} on vector space $V$, and also denote  \[V_\Lambda^0=\set{v\in V\given \lambda(v)= 0\text{ for all }\lambda\in\Lambda}=\cap_{\lambda\in\Lambda}\ker\lambda.\]
Given such a corner-defining system, a \textit{chart with corners} on $X$ at $p\in X$ is a map $\varphi\colon U\rightarrow V_\Lambda$ where $U\subset X$ is an open neighborhood of $p$, $\varphi(U)\subset V_\Lambda$ is an open neighborhood of $0$, and $\varphi$ is a homeomorphism onto its image with $\varphi(p)=0$. Such a chart is said to be $n$\textit{-dimensional}, and the point $p$ is said to have \textit{index} $|\Lambda|$, for which we will write $\ind_X(p)=|\Lambda|$. 

We say that $X$ is an $n$-dimensional \textit{manifold with corners} if there is an $n$-dimensional chart with corners around every point, and any two such charts $\varphi$ and $\varphi'$ are \textit{compatible}, i.e.\ $\varphi(U\cap U')\subset V_\Lambda$ and $\varphi'(U\cap U')\subset V_{\Lambda'}$ are open subsets, and $\varphi'\circ\varphi^{-1}\colon \varphi(U\cap U')\rightarrow\varphi'(U\cap U')$ is a diffeomorphism.\footnote{Denote $\R^n_k=\R^{n-k}\times[0,\infty)^k$. A map $ U\rightarrow V$ between open subsets of $\R^n_k$ and $\R^m_l$ is \textit{smooth} at $p\in U$, if it admits a smooth extension to an open neighborhood of $p$ in $\R^n$.} Such a collection of charts with corners is called an \textit{atlas with corners}. On a manifold with corners, the map $\ind_X\colon X\rightarrow \mathbb{N}_0$ is well-defined by boundary invariance theorem found in \cite[Theorem 1.2.12]{corners}. 

We will use the following notation regarding the elementary terms for manifolds with corners: the $k$\textit{-boundary} of $X$ is denoted by $\partial^kX=\set{p\in X\given \ind_X(p)\geq k}$, and we say that $X$ is a \textit{manifold with boundary} if $\partial X:=\partial^1 X\neq \emptyset$ and $\partial^2X=\emptyset$. The $k$\textit{-stratum} of $X$ is its subspace $S^k(X)=\set{p\in X\given \ind_X(p)=k}$. The set of connected components of the $k$-stratum is denoted by $\mathcal S^k(X)$, and the set of $k$\textit{-faces} of $X$ is the family 
\[\mathcal F^k(X)=\set*{\bar S\given S\in \mathcal S^k(X)}\]
of topological closures in $X$ of the the components of the $k$-stratum. We note that the $k$-stratum $S^k(X)$ can be given the following canonical differential structure induced by $X$: if $p\in S^k(X)$ and $\varphi\colon U\rightarrow V_\Lambda$ is a corner chart at $p$, then the restriction 
\[
\varphi|_{U\cap S^k(X)}\colon U\cap S^k(X)\rightarrow V_\Lambda^0
\] is a corner chart at $p$ on $S^k(X)$ since there holds $\varphi(U\cap S^k(X))=\varphi(U)\cap V_\Lambda^0$. With this structure, $S^k(X)$ becomes a manifold without boundary of dimension $n-k$.

A subspace $X'\subset X$ is said to be a \textit{submanifold} of $X$, if for any $p\in X'$ there is a chart with corners $\varphi\colon U\rightarrow V_\Lambda$ at $p$ on $X$, and a linear subspace $V'$ together with a corner-defining system $\Lambda'$ on $V'$, such that $\varphi(U\cap X')=\varphi(U)\cap V'_{\Lambda'}$ and this is an open subset of $V'_{\Lambda'}$. Such a chart is said to be \textit{adapted} to $X''$ by means of $(V',{\Lambda'})$, and gives us a way of defining an intrinsic atlas with corners on $X'$. If there holds $\partial X'=\partial X\cap X'$, we say that $X'\subset X$ is a \textit{neat} submanifold.

In the context of manifolds with corners, it is more useful to use non-infinitesimal notions of submersivity and transversality. 
\begin{defn}
\label{defn:top_sub}
A smooth map $f\colon X\rightarrow X'$ between manifolds with corners is a \textit{topological submersion} at $p\in X$, if there is an open neighborhood $U$ of $f(p)$ in $X'$ and a smooth map $\sigma\colon U\rightarrow X$ such that $\sigma(f(p))=p$ and $f\circ\sigma=\id_{U}$. If this holds for all $p\in X$, we just say that $f\colon X\rightarrow X'$ is a topological submersion.
\end{defn}
The notion of a topological submersion is stronger than the usual infinitesimal one which can be seen by differentiating the equality $f\circ\sigma=\id_U$. Moreover, in the context of boundaryless manifolds they are equivalent, which is an easy consequence of the usual rank theorem.

\begin{defn}
Suppose that $f\colon X\rightarrow X'$ is a smooth map between manifolds with corners, and $X''\subset X'$ is a submanifold. The map  $f$ is \textit{topologically transversal} to $X''$ at $p\in X$, written symbolically as \[f\pitchfork_p X'',\] if either $f(p)\notin X''$, or there is a chart $\varphi'\colon U'\rightarrow V'_{\Lambda'}$ on $X'$ at $p$ adapted to $X''$ by means of $(V'',{\Lambda''})$ and an open neighborhood $U$ of $p$ in $X$, such that $f(U)\subset U'$ and
\[\begin{tikzcd}
	\tau\colon U & {U'} & {\varphi'(U')} & {(V''\oplus W'')_{L^*\Lambda'}} & {W''}
	\arrow["{f|_U}", from=1-1, to=1-2]
	\arrow["{\varphi'}", "\approx"', from=1-2, to=1-3]
	\arrow["{L^{-1}}", from=1-3, to=1-4]
	\arrow["{\mathrm{pr}_2}", from=1-4, to=1-5]
\end{tikzcd}\]
is a topological submersion at $p$, where $W''$ is any complementary subspace to $V''$ in $V'$, and $L\colon V''\oplus W''\rightarrow V'$ is the linear isomorphism given by $L(v,w)=v+w$. If for all $p\in X$ we have $f\pitchfork_p X$, we just write $f\pitchfork X''$ and say $f$ is topologically transversal to $X''$. 
\end{defn}
\noindent Note that $f\pitchfork_p X''$ implies the usual infinitesimal transversality condition
\begin{align}
\label{eq:inf_trans}
\d f_p(T_p X)+ T_{f(p)}X''=T_{f(p)}X'.
\end{align}
Indeed, if $v'\in T_{f(p)}X'$, there holds
\[
v'=\d f_x(v)+\d(\varphi')_p^{-1}L\big(\pr_1 L^{-1}\d(\varphi')_p(v'-\d f_x(v)),0\big)
\]
where $v\in T_p X$ with $\d \tau_p (v)=\mathrm{pr}_2  L^{-1}\d(\varphi')_p(v')$ exists since $\tau$ is a submersion at $p$. We will show in Proposition \ref{prop:transversality_char} that \eqref{eq:inf_trans} implies $f\pitchfork_p X''$ in case $X$ has no boundary.

We now state the main transversality theorem for manifolds with corners, which is a generalization of the same result for boundaryless manifolds.
\begin{prop}
\label{prop:transversality}
Let $f\colon X\rightarrow X'$ be a smooth map between manifolds with corners and let $X''\subset X'$ be a neat submanifold. If $f\pitchfork X''$, then $f^{-1}(X'')\subset X$ is a submanifold with $\codim_X f^{-1}(X)=\codim_{X'}X''$, whose tangent bundle is
\[T(f^{-1}(X''))=(\d f)^{-1}(TX'').\]
Moreover, $f^{-1}(X'')\subset X$ is totally neat, i.e.\ $S^k(f^{-1}(X))=f^{-1}(X'')\cap  S^k(X)$ for any $k\geq 0$. In particular, $f^{-1}(X)$ is a neat submanifold of $X$.
\end{prop}
\begin{proof}
\cite[Proposition 7.1.14]{corners}.
\end{proof}
In the context of Lie categories with nonempty regular boundaries, this result enables us to show that the set of composable morphisms has a structure of a smooth manifold. Let us show how.
\begin{lem}
\label{lem:bdry_topsub}
Let $f\colon X\rightarrow X'$ be a smooth map from a manifold $X$ with boundary to a boundaryless manifold $X'$, such that both $f$ and $\partial f$ are submersions. Then $f$ is a topological submersion.
\end{lem}
\begin{proof}
We want to show $f$ is a topological submersion at any $p\in X$. If $p\in \Int X$ this follows from the usual rank theorem used on $f|_{\Int X}$, and if $p\in \partial X$ it follows from the usual rank theorem used on $f|_{\partial X}$.
\end{proof}
\begin{lem}
\label{lem:topsub_transversal}
If $f\colon X\rightarrow X'$ is a smooth topological submersion between manifolds with corners and $\partial X'=\emptyset$, then $f\pitchfork X''$ holds for any submanifold $X''\subset X'$.
\end{lem}
\begin{proof}
Suppose $f(p)\in X''$ and let $\varphi'\colon U'\rightarrow V'$ be a chart on $X'$ at $f(p)$ adapted to $X''$ by means of $(V'',{\Lambda''})$. Let $W''$ be any complementary subspace to $V''$ in $V$; continuity of $f$ ensures there is a neighborhood $U\subset X$ of $p$ such that $f(U)\subset U'$, and now the composition
\[\begin{tikzcd}
	U & {U'} & {\varphi'(U')} & {V''\oplus W''} & {W''}
	\arrow["{f|_U}", from=1-1, to=1-2]
	\arrow["{\varphi'}", "\approx"', from=1-2, to=1-3]
	\arrow["{L^{-1}}", from=1-3, to=1-4]
	\arrow["{\mathrm{pr}_2}", from=1-4, to=1-5]
\end{tikzcd}\]
is a topological submersion as a composition of topological submersions $f|_U$ and $\pr_2 L^{-1}\varphi'$.
\end{proof}
\begin{cor}
\label{cor:fibred_prod}
Let $X$ and $Y$ be manifolds with boundaries and $Z$ a manifold without boundary. If $f\colon X\rightarrow Z$ and $g\colon Y\rightarrow Z$ are smooth maps such that $f,f|_{\partial X}$ and g,$g|_{\partial Y}$ are submersions, then $X{\tensor*[_f]{\times}{_g}}Y=(f\times g)^{-1}(\Delta_{Z})\subset X\times Y$ is a submanifold with tangent space at $(p,q)$ equal to
\[
T_{(p,q)}(X{\tensor*[_f]{\times}{_g}}Y)=\set{(v,w)\in T_p X\oplus T_q Y\given \d f(v)=\d g(w)},
\]
and its boundary is $\partial(X{\tensor*[_f]{\times}{_g}}Y)=(X{\tensor*[_f]{\times}{_g}}Y )\cap (\partial X\times Y\cup X\times\partial Y)$.
\end{cor}
\begin{proof}
Lemma \ref{lem:bdry_topsub} ensures $f$ and $g$ are topological submersions, and it is easy to check that $f\times g\colon X\times Y\rightarrow Z\times Z$ is also a smooth topological submersion. Now Lemma \ref{lem:topsub_transversal} implies $(f\times g)\pitchfork\Delta_Z$ since $Z$ is boundaryless, and Proposition \ref{prop:transversality} finishes the proof since $\Delta_Z\subset Z\times Z$ is trivially a neat submanifold.
\end{proof}
The following proposition shows in particular that in the case of boundaryless manifolds, the usual infinitesimal notion of transversality is equivalent to the one above.
\begin{prop}
\label{prop:transversality_char}
Let $f\colon X\rightarrow X'$ be a smooth map between manifolds with corners and $X''\subset X'$ a submanifold. For any $p\in f^{-1}(X'')$ with $\ind_X(p)=k$, the following statements are equivalent.
\begin{enumerate}
\item $\Im\d (f|_{S^k(X)})_p+ T_{f(p)}X''=T_{f(p)}X'.$
\item $f|_{S^k(X)}\pitchfork_p X''$.
\item $f\pitchfork_p X''$.
\end{enumerate}
\end{prop}
\begin{proof}
To show $(i)\Rightarrow (ii)$, take a chart $\varphi'\colon U'\rightarrow V'_{\Lambda'}$ on $X'$ at $f(p)$ adapted to $X''$ by means of $(V'',{\Lambda''})$, take a complementary subspace $W''$ to $V''$ of $V'$, and consider the map
\[\begin{tikzcd}
	\tau^k\colon U & {U'} & {\varphi'(U')} & {(V''\oplus W'')_{L^*\Lambda'}} & {W''}
	\arrow["{f|_U}", from=1-1, to=1-2]
	\arrow["{\varphi'}", "\approx"', from=1-2, to=1-3]
	\arrow["{L^{-1}}", from=1-3, to=1-4]
	\arrow["{\mathrm{pr}_2}", from=1-4, to=1-5]
\end{tikzcd}\]
where $U$ is the domain of a chart neighborhood $\varphi\colon U\rightarrow V$ in $S^k(X)$ of $p$; by continuity we may assume $f(U)\subset U'$. Since $\partial S^k(X)$ is boundaryless, it is enough to show that $\tau^k$ is a submersion. To that end, take any $w\in T_0 W''\cong W''$ and then $u:=\d(\varphi')_p^{-1}L(0,w)\in T_{f(p)}X'$, so by assumption there exist $v\in T_p X$ and $v''\in T_{f(p)}X''$ such that $u=\d f_p(v)+v''$. Since $L^{-1}\d(\varphi')_p(v'')\in V''$, we get
\[
\d h_p(v)=\pr_2L^{-1}\d(\varphi')(u-v'')=w.
\]
Conversely, $(ii)\Rightarrow (i)$ follows from the fact that topological transversality implies transversality.

For the implication $(ii)\Rightarrow (iii)$, note that if $Z\xrightarrow{g} Z'\xrightarrow{h} Z''$ are smooth maps such that $h\circ g$ is a topological submersion at $p\in Z$, then it is easy to see $h$ is a topological submersion at $g(p)$. Use this result for the composition $S^k(X)\hookrightarrow X\xrightarrow \tau W''$ which equals $\tau^k$. The converse implication $(iii)\Rightarrow (ii)$ is a consequence of the following lemma used on the map $\tau$.
\end{proof}

\begin{lem}
If $f\colon X\rightarrow X'$ is a smooth map between manifolds with corners, which is a topological submersion at $p$ and there holds $f(p)\in \Int (X')$, then $f|_{S^k(X)}$ is a submersion at $p$, where $k=\ind_X(p)$.
\end{lem}
\begin{proof}
Let $\varphi\colon U\rightarrow V_{\Lambda}$ be a chart on $X$ at $p$ and $\varphi'\colon X'\rightarrow V'$ on $X'$ at $f(p)$. Note that $\varphi(U\cap S^k(X))=\varphi(U)\cap V_{\Lambda}^0$.

Let $v'\in T_{f(p)}X'$ be arbitrary. By assumption, there is a neighborhood $U'$ of $f(p)$ and a smooth map $\sigma\colon U'\rightarrow U\subset X$ such that $\sigma(f(p))=p$ and $f\circ \sigma =\id_{U'}$. Since $f(p)\in\Int X'$, there is a smooth path $\gamma\colon(-\varepsilon,\varepsilon)\rightarrow U'$ with $\dot\gamma(0)=v'$, and now consider the path $\delta=\varphi\circ\sigma\circ\gamma\colon(-\varepsilon,\varepsilon)\rightarrow \varphi(U)\subset V_{\Lambda}$. Since $\delta(0)=0$ and $\delta$ is defined on an open interval, we must have that $\delta$ maps into $V_{\Lambda}^0$, so $\dot\delta(0)\in T_0V_{\Lambda}^0\cong V_{\Lambda}^0$. Taking $v=\d(\varphi^{-1})_0(\dot\delta(0))$ hence yields $v\in T_p(S^k(X))$ with $\d f_p(v)=v'$.
\end{proof}









\Addresses


\begin{bibdiv}
\begin{biblist}

\bib{ehr1959}{article}{
 author = {Ehresmann, Charles}
 title = {Cat{\'e}gories topologiques et cat{\'e}gories diff{\'e}rentiables},
 year = {1959}
 language = {French}
 pages={137-150}
 journal = {Centre {Belge} {Rech}. {Math}., Colloque de Géométrie différentielle globale}
 Keywords = {58A05,18F99}
 Zbl = {0205.28202}
}

\bib{mackenzie}{book}{
    AUTHOR = {Mackenzie, Kirill C. H.},
     TITLE = {General theory of {L}ie groupoids and {L}ie algebroids},
    SERIES = {London Mathematical Society Lecture Note Series},
    VOLUME = {213},
 PUBLISHER = {Cambridge University Press, Cambridge},
      YEAR = {2005},
     PAGES = {xxxviii+501},
      ISBN = {978-0-521-49928-3; 0-521-49928-3},
   MRCLASS = {58H05 (53D17)},
  MRNUMBER = {2157566},
MRREVIEWER = {Rui Loja Fernandes},
       DOI = {10.1017/CBO9781107325883},
       URL = {https://doi.org/10.1017/CBO9781107325883},
}

\bib{catsquantum}{book}{
    AUTHOR = {Heunen, Chris},
    AUTHOR = {Vicary, Jamie},
     TITLE = {Categories for quantum theory},
    SERIES = {Oxford Graduate Texts in Mathematics},
    VOLUME = {28},
      NOTE = {An introduction},
 PUBLISHER = {Oxford University Press, Oxford},
      YEAR = {2019},
     PAGES = {xii+324},
      ISBN = {978-0-19-873961-6; 978-0-19-873962-3},
   MRCLASS = {18-02 (18Mxx 81P10)},
  MRNUMBER = {3971584},
MRREVIEWER = {Shawn Xingshan Cui},
       DOI = {10.1093/oso/9780198739623.001.0001},
       URL = {https://doi.org/10.1093/oso/9780198739623.001.0001},
}


\bib{difftop}{book}{
    AUTHOR = {Victor Guillemin},
    AUTHOR = {Alan Pollack}
     TITLE = {Differential topology},
      NOTE = {Reprint of the 1974 original},
 PUBLISHER = {AMS Chelsea Publishing, Providence, RI},
      YEAR = {2010},
     PAGES = {xviii+224},
      ISBN = {978-0-8218-5193-7},
   MRCLASS = {58-01 (57-01)},
  MRNUMBER = {2680546},
       DOI = {10.1090/chel/370},
       URL = {https://doi.org/10.1090/chel/370},
}

\bib{dynamics}{article}{
    AUTHOR = {Cabrera, Alejandro},
    AUTHOR = {del Hoyo, Matias},
    AUTHOR = {Pujals, Enrique},
     TITLE = {Discrete dynamics and differentiable stacks},
   JOURNAL = {Rev. Mat. Iberoam.},
  FJOURNAL = {Revista Matem\'{a}tica Iberoamericana},
    VOLUME = {36},
      YEAR = {2020},
    NUMBER = {7},
     PAGES = {2121--2146},
      ISSN = {0213-2230},
   MRCLASS = {37A99 (22A22 58J42)},
  MRNUMBER = {4163995},
       DOI = {10.4171/rmi/1194},
       URL = {https://doi.org/10.4171/rmi/1194},
}

\bib{corners}{book}{
    AUTHOR = {Margalef Roig, Juan},
    AUTHOR = {Outerelo Dom\'{\i}nguez, Enrique},
     TITLE = {Differential topology},
    SERIES = {North-Holland Mathematics Studies},
    VOLUME = {173},
      NOTE = {With a preface by Peter W. Michor},
 PUBLISHER = {North-Holland Publishing Co., Amsterdam},
      YEAR = {1992},
     PAGES = {xvi+603},
      ISBN = {0-444-88434-3},
   MRCLASS = {58Bxx (58-01 58A05 58D15)},
  MRNUMBER = {1173211},
MRREVIEWER = {Wies\l aw Sasin},
}


\bib{sussman}{article}{
    AUTHOR = {Sussmann, H\'{e}ctor J.},
     TITLE = {Orbits of families of vector fields and integrability of
              distributions},
   JOURNAL = {Trans. Amer. Math. Soc.},
  FJOURNAL = {Transactions of the American Mathematical Society},
    VOLUME = {180},
      YEAR = {1973},
     PAGES = {171--188},
      ISSN = {0002-9947},
   MRCLASS = {58A30 (53C10)},
  MRNUMBER = {321133},
MRREVIEWER = {A. Morimoto},
       DOI = {10.2307/1996660},
       URL = {https://doi.org/10.2307/1996660},
}

\bib{lee}{book}{
    AUTHOR = {Lee, John M.},
     TITLE = {Introduction to smooth manifolds},
    SERIES = {Graduate Texts in Mathematics},
    VOLUME = {218},
   EDITION = {Second},
 PUBLISHER = {Springer, New York},
      YEAR = {2013},
     PAGES = {xvi+708},
      ISBN = {978-1-4419-9981-8},
   MRCLASS = {58-01 (53-01 57-01)},
  MRNUMBER = {2954043},
}

\bib{baez2011}{article}{
    AUTHOR = {Baez, John C.},
    AUTHOR = {Fritz, Tobias},
    AUTHOR = {Leinster, Tom},
     TITLE = {A characterization of entropy in terms of information loss},
   JOURNAL = {Entropy},
  FJOURNAL = {Entropy. An International and Interdisciplinary Journal of
              Entropy and Information Studies},
    VOLUME = {13},
      YEAR = {2011},
    NUMBER = {11},
     PAGES = {1945--1957},
   MRCLASS = {94A17 (62B10)},
  MRNUMBER = {2868846},
MRREVIEWER = {Konstantinos Zografos},
       DOI = {10.3390/e13111945},
       URL = {https://doi.org/10.3390/e13111945},
}


\bib{poisson}{book}{
	AUTHOR = {Crainic, Marius},
	AUTHOR = {Fernandes, Rui Loja},
	AUTHOR = {M\u{a}rcu\c{t}, Ioan}
     TITLE = {Lectures on {P}oisson geometry},
    SERIES = {Graduate Studies in Mathematics},
    VOLUME = {217},
 PUBLISHER = {American Mathematical Society, Providence, RI},
      YEAR = {2021},
     PAGES = {xix+479},
      ISBN = {978-1-4704-6430-1},
   MRCLASS = {53D17 (22Exx 53Dxx 58Hxx)},
  MRNUMBER = {4328925},
MRREVIEWER = {Jochen Merker},
       DOI = {10.1090/gsm/217},
       URL = {https://doi.org/10.1090/gsm/217},
}

\bib{mec}{unpublished}{
	AUTHOR = {Fernandes, Rui Loja},
	AUTHOR = {M\u{a}rcu\c{t}, Ioan}
     TITLE = {Multiplicative Ehresmann Connections},
      YEAR = {2022},      
       note = {Preprint available at \url{https://arxiv.org/abs/2204.08507}},
}








  \end{biblist}
\end{bibdiv}
\end{document}